\documentclass[12pt]{amsart}

\usepackage{graphicx}
 \usepackage[leqno]{amsmath}
\usepackage{amssymb}
\usepackage[all]{xy}
\usepackage{graphicx}
\usepackage[usenames,dvipsnames]{color}

\newtheorem{thm}{Theorem}[section]

\newtheorem{lemma}[thm]{Lemma}
\newtheorem{cor}[thm]{Corollary}
\newtheorem{prop}[thm]{Proposition}

\theoremstyle{definition}
\newtheorem{definition}[thm]{Definition}
\newtheorem{example}[thm]{Example}

\newtheorem{conjecture}[thm]{Conjecture}

\title{Paper}

\date{September 27, 2012}

\begin{document}

\title[Fractional Dehn Twists]
{Fractional Dehn Twists in Knot Theory and Contact Topology}

\author[Kazez]{William H. Kazez}
\address{Department of Mathematics, University of Georgia, Athens, GA 30602}
\email{will@math.uga.edu}

\author[Roberts]{Rachel Roberts}
\address{Department of Mathematics, Washington University, St. Louis, MO 63130}
\email{roberts@math.wustl.edu}

\keywords{fractional Dehn twist, overtwisted, contact structure, open book decomposition, fibered
link, surface automorphism}

\subjclass[2000]{Primary 57M50; Secondary 53D10.}

\thanks{WHK supported by NSF grant DMS-0073029, RR supported by NSF grant DMS-0312442.}

\begin{abstract} Fractional Dehn twists give a measure of the difference between the relative isotopy class of a homeomorphism of a bordered surface and the Thurston representative of its free isotopy class.  We show how to estimate and compute these invariants.  We discuss the the relationship of our work to stabilization problems in classical knot theory, general open book decompositions, and contact topology.  We include an elementary characterization of overtwistedness for contact structures described by open book decompositions.
\end{abstract}


\maketitle

\section{Introduction}

Given an automorphism $h$ of a bordered surface $S$ that is the identity map on $\partial S$, the fractional Dehn twist coefficient of $h$ with respect to a component $C$ of $\partial S$ is a rational number.  It measures the amount of rotation that takes place about $C$ when isotoping $h$ to its Thurston representative.  See Section~\ref{basics} for precise definitions of $c(h)$.  

This concept was first studied by Gabai and Oertel in \cite{GO}, where it is called the degeneracy slope and was applied to problems related to  Dehn fillings of manifolds containing essential laminations.  It is also used, with different coordinates, by Roberts \cite{Ro1,Ro2} to describe surgeries on fibred knots in which she can construct taut foliations.  Most recently it was used by Honda, Kazez, and Mati\'c \cite{HKM1,HKM2} to quantify the concept of right-veering that arises in contact topology.

Section~\ref{basics} describes techniques for estimating and computing $c(h)$ and gives some examples.  We prove a general result, Theorem~\ref{stabilizedc(h)}, that $c(h)\in[0,1/2]$ if $h$ is a stabilization, and begin to develop properties of the important case when $c(h)=1/2$.

In Section~\ref{pA} we construct a class of examples, motivated by an example of Gabai's \cite{Ga3}, for which $c(h)=1/2$.  The idea is to replace a fibred knot with its $(2,1)$-cable.  This creates a fibred knot with reducible monodromy and fractional Dehn twist coefficient $1/2$.  Following this with surgery on an unknotted curve in the fibre of the $(2,1)$-cable can produce examples of fibred knots with pseudo-Anosov monodromy that also have fractional Dehn twist coefficient 1/2.

 Our interest in producing examples of fibred knots in $S^3$ with fractional Dehn twist coefficient 1/2 arises from  Conjecture~\ref{main1}.  This conjecture states that a knot in $S^3$ with fractional Dehn twist coefficient 1/2 can not be destabilized.  Section 4 includes a brief history of this conjecture.

The relevance of  Conjecture~\ref{main1} in contact topology is discussed in Section~\ref{applications}.  Denote by $\xi$ the contact structure compatible with a pair $(S,h)$.  If true, this conjecture would provide the simplest known counterexamples to a conjecture of Honda, Kazez, and Mati\'c \cite{HKM1} by producing fibred knots in $S^3$ that are not stabilizations, are right-veering, yet are still overtwisted.  The section includes some context for the conjecture of \cite{HKM1}, references to earlier counterexamples for links bounding planar surfaces by Lekili \cite{Lekili}, Lisca \cite{Lisca}, and Ito and Kawamuro \cite{IK}, and some remarks on work of Colin and Honda \cite{CoHo}.  The technique we use for recognizing that a contact structure is overtwisted is completely elementary: we exhibit an overtwisted disk by finding an unknotted, untwisted curve on the Seifert surface of the knot.

We would like to thank  David Gabai and Jennifer Schultens for helpful conversations.

\section{Computing fractional Dehn twist coefficients}\label{basics}

\bigskip

We recall Thurston's classification of surface automorphisms.

\begin{thm}\label{Thurston} \cite{Th, CB} Let $S$ be an oriented hyperbolic surface with geodesic boundary, and let $h\in Aut(S,\partial S)$. Then $h$ is freely isotopic to either 

(1) a pseudo-Anosov homeomorphism $\phi$ that preserves a pair of geodesic laminations $\lambda^s$ and $\lambda^u$,

(2) a periodic homeomorphism $\phi$, in which case there is a hyperbolic metric for which $S$ has geodesic boundary and such that $\phi$ is an isometry of $S$, or

(3) a reducible homeomorphism $h'$ that fixes, setwise, a maximal collection of simple closed geodesic curves $\{C_j\}$ in $S$.
\end{thm}

To avoid overlap in the cases, we refer to a map as reducible only if it is not periodic.  Given a reducible map, splitting $S$ along $\cup_jC_j$ gives a collection of surfaces $S_1,\dots, S_n \subset S$ with geodesic boundary that are permuted by $h'$.  Choose an integer $n_i$ so that $(h')^{n_i}$ maps $S_i$ to itself.  Maximality of $\{C_j\}$ implies that applying Thurston's classification theorem to $(h')^{n_i} \in Aut(S_i, \partial S_i)$ produces either a pseudo-Anosov or periodic representative.

Our primary focus will be in describing homeomorphisms from the point of view of a single boundary component $C\subset \partial S$ that is fixed pointwise by $h$.  If such a map is reducible, let $S_0$ be the subsurface $S_i$ (as above) of $S$ which contains $C$.  Then $h'(S_0)=S_0$, and we let $\phi_0$ be the pseudo-Anosov or periodic representative of $h'|S_0$.  
With this notation, we make the following definition.  

\begin{definition}\label{Thurstonrep}  A map $\phi \in Aut(S, \partial S)$, freely isotopic to $h$, is called a {\it Thurston representative} of $h$, if it is pseudo-Anosov, periodic, or reducible with $\phi|S_0 =\phi_0$.
\end{definition}

Consider the open book determined by the data $(S,h)$ and the suspension flow of $\phi$.  This flow, when restricted to the component corresponding to $C$ in the complement of the binding, necessarily has periodic orbits.  Let $\gamma$ be one such, and write $$\gamma = p\lambda+q\nu$$ where $\lambda = \partial S$, $\nu$ is the meridian, and $p, q$ are relatively prime integers. The {\sl fractional Dehn twist coefficient} of $h$ with respect to a component $C$ of $\partial S$ \cite{HKM1} is given by $$c(h)=p/q\, .$$

In the context of surgery on knots in $S^3$, the reciprocal quantity, $q/p$, is called the {\sl degeneracy slope} by Gabai~\cite{Ga3}. See also \cite{GO} for motivating work describing the location of cusps from the point of view of a toroidal boundary component in a 3-manifold carrying an essential lamination.  Fractional Dehn twists, though with different coordinates, also play an important role in Roberts' constructions of taut foliations \cite{Ro1,Ro2}.  See \cite{HKM2} for a description of the change in coordinates.

An alternate description of the fractional Dehn twist coefficient, which is meant to emphasize that $h$ looks like a fraction of a Dehn twist has been applied to the boundary of $S$ is given as follows. First extend $h$ by the identity map to a homeomorphism of $F=S \cup (\partial S \times I)$ where the annuli added to the boundary of $S$ are glued along $\partial S \times \{1\}$.  Next, freely isotop $h$ to its Thurston representative, $\phi$, on $S$, and extend this representative across $\partial S \times I$ so that the resulting homeomorphism of $F$ is isotopic to $id \cup h$ relative to $\partial F$. The extension can be chosen to be a shearing map, that is, to be a straight-line homotopy when lifted to the universal cover of $\partial S \times I$ thought of as a subset of the plane.

The periodic orbit $\gamma$ determines a collection of points $\{x_0, \dots, x_{q-1}\} \subset C$ labelled cyclically and ordered compatibly with the induced orientation on $C$.  The shearing homeomorphism that extends $\phi$ across $C \times I$ maps $x_0 \times I$ to an arc that may spiral more than once around $C$ in either direction.  Lifting to the the universal cover and indexing with integers rather than working mod $q$, it must end at some $x_p \times \{1\}\in \partial \tilde S$.  It then follows that $c(h)=p/q$.

We now describe how to compute $c(h)$ with respect to $C$ by computing the action of $h$ on some arcs. Let $\alpha$ be an oriented, properly embedded arc in $S$ starting on $C$.  Define $i_h(\alpha)$ to measure the number of intersections of $\alpha$ and $h(\alpha)$ which occur after their common initial point but in an annular neighborhood of $C$.  

\begin{figure}[ht]
  	\centering
   	\includegraphics[width=1.8truein]{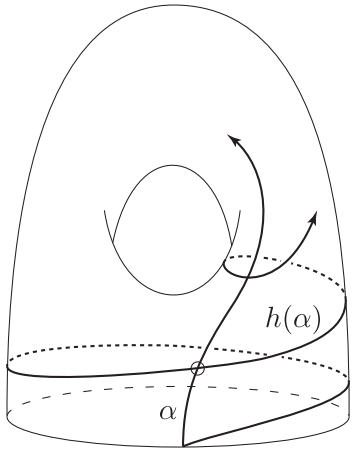} 
	\caption{In this example $i_h(\alpha)=1$.}
	\label{initialintersection}
\end{figure}

Topologically, this means first isotoping the arcs $\alpha$ and $h(\alpha)$, relative to their boundaries, so they intersect minimally. Then  $i_h(\alpha)$ is defined to be a signed count of the number of points, $x$, in the intersection of the interiors of $h(\alpha)$ and $\alpha$ with the property that the union of the initial segments of these arcs, up to $x$, is contained in an annular neighborhood of $C$.

It is also useful to define $i_h(\alpha)$ geometrically by putting the arcs in a canonical position.  We can think of arcs as living in $F=S \cup (\partial S \times I)$ a hyperbolic surface with geodesic boundary together with annuli extending the boundary components.  Denote by $\overline\alpha$ the unique geodesic arc in the free isotopy class of $\alpha$ that is orthogonal to the initial and terminal boundary components of $S$.  Now let $\alpha_0$ and $\alpha_1$ be monotonically spiraling arcs in $\partial S \times I$ such that $\alpha_0 * \overline\alpha*\alpha_1$ is an arc isotopic, relative to its endpoints, to $\alpha$.  Write $h(\alpha)_0 * \overline {h(\alpha)}*h(\alpha)_1$ for the canonical form of $h(\alpha)$.

The geometric description of $i_h(\alpha)$ depends on properties of its Thurs\-ton representative, $\phi$.  In all cases $\overline{\phi(\alpha)} = \overline {h(\alpha)}$.  

If $\phi$ is pseudo-Anosov, it fixes no compact geodesic, and hence $\overline {h(\alpha)}$ and $\overline{\alpha}$ intersect minimally.  It follows that $i_h(\alpha)$ is the number of intersections of the interiors of $h(\alpha)_0$ and $\alpha_0$ counted with sign; namely, 
$$i_h(\alpha) = \langle int \, h(\alpha)_0, int\,\alpha_0\rangle\, .$$

The same formula for $i_h(\alpha)$ applies when $\phi$ is periodic, provided that $\overline {h(\alpha)}\ne\overline{\alpha}$.  In the case  that $\overline {h(\alpha)}=\overline{\alpha}$,  $\phi$ is an orientation preserving isometry that agrees with the identity map at a point of $\partial S$, and hence  is the identity map. It follows that $h$ is the composition of Dehn twists along the boundary components of $S$. 

In the case that $\phi$ is reducible let $h', \phi_0$ and $S_0$ be as in Definition~\ref{Thurstonrep}.  If $\alpha\subset S_0$, then the formulas given above can be applied to compute $i_{h'}(\alpha)=i_h(\alpha)$.  Otherwise, the geometric description breaks into cases. Although an analysis of these cases is straightforward, it seems to yield little insight, and so we do not include this here.

\begin{definition} \cite{HKM1} Let $\alpha$ and $\beta$ be properly embedded oriented arcs in $S$ with a common initial point.  Suppose, without changing notation, that they have been isotoped, while fixing endpoints, to intersect transversely in the minimum possible number of points.  We say $\alpha$ is {\sl to the right of} $\beta$ if the frame consisting of the tangent vector to $\alpha$ followed by the tangent vector to $\beta$ agrees with the orientation on $S$. Similarly, we say $\alpha$ is {\sl to the left of} $\beta$ if the frame consisting of the tangent vector to $\beta$ followed by the tangent vector to $\alpha$ gives the orientation on $S$. This definition implies that every arc is both to the left and right of itself.
\end{definition}
See \cite{HKM1} for several equivalent definitions.

\begin{definition} \cite{HKM1} A homeomorphism $h$ that restricts to the identity map on $\partial S$ is called {\sl right-veering at $C$} if for every embedded oriented arc $\alpha$ with initial endpoint on $C$, $h(\alpha)$ is to the right of $\alpha$. Similarly, a homeomorphism $h$ that restricts to the identity map on $\partial S$ is called {\sl left-veering at C} if for every embedded oriented arc $\alpha$ with initial endpoint on $C$, $h(\alpha)$ is to the left of $\alpha$.
\end{definition}

The first step in using $i_h$ to compute $c(h)$ is the following proposition.

\begin{prop}\cite{HKM1}\label{RV}  Let $h\in Aut(S,\partial S)$. Then $h$ is right-veering if and only if $c(h)>0$ for every component of $\partial S$, and similarly $h$ is left-veering if and only if $c(h)<0$ for every component of $\partial S$.
\end{prop}

We now return our focus to a component $C$ of $\partial S$. In what follows, $c(h)$ denotes the fractional Dehn twist coefficient with respect to $C$ and properly embedded oriented arcs have initial point on $C$ as necessary.

\begin{cor}\label{c(h)<>0}If $\alpha$ is a properly embedded oriented arc in $S$, then $h(\alpha)$ is to the right of $\alpha$ implies $c(h) \ge 0$, and $h(\alpha)$ is to the left of $\alpha$ implies $c(h) \le 0$. \qed
\end{cor}

This gives a quick way to show that $c(h)=0$ by producing two arcs, one moved to the right by $h$ and the other moved to the left by $h$.  See Example~\ref{Example_8_20}. 
If $c(h)=0$ it is always possible to find such a pair of arcs. Note that if $\phi$ is periodic or $\phi$ is reducible with periodic $\phi_0$, then we can choose a single arc
$\alpha$ fixed by $h$ (and hence moved both to the left and to the right by $h$). 

For $h$ with $c(h)=\frac{p}{q}$ with $p \neq 0$, the next proposition shows how to produce a homeomorphism $g$ for which computing $c(g)=0$ is equivalent to showing $c(h)=\frac{p}{q}$.

\begin{prop}\label{formula} If $g=\tau_C^{-p}(h^q)$, then $c(g)=qc(h)-p$. \qed
\end{prop}

Propositions~\ref{formula}, \ref{rightveeringi}, and \ref{leftveeringi} make it possible to estimate $c(h)$ experimentally.

\begin{example}\label{Example_8_20} Figure~\ref{8_20} shows the knot $8_{20}$, its Seifert surface, $S$, and two curves living on $S$. This is a fibred knot with pseudo-Anosov monodromy $h$.  The curves shown are the cores of left and right Hopf bands.  It follows that the indicated arcs $a$ and $b$ in $S$ are moved to the right and to the left, respectively.  By Corollary~\ref{c(h)<>0}, $c(h)=0$.
\end{example}

\begin{figure}[ht]
  	\centering
   	\includegraphics[width=2truein]{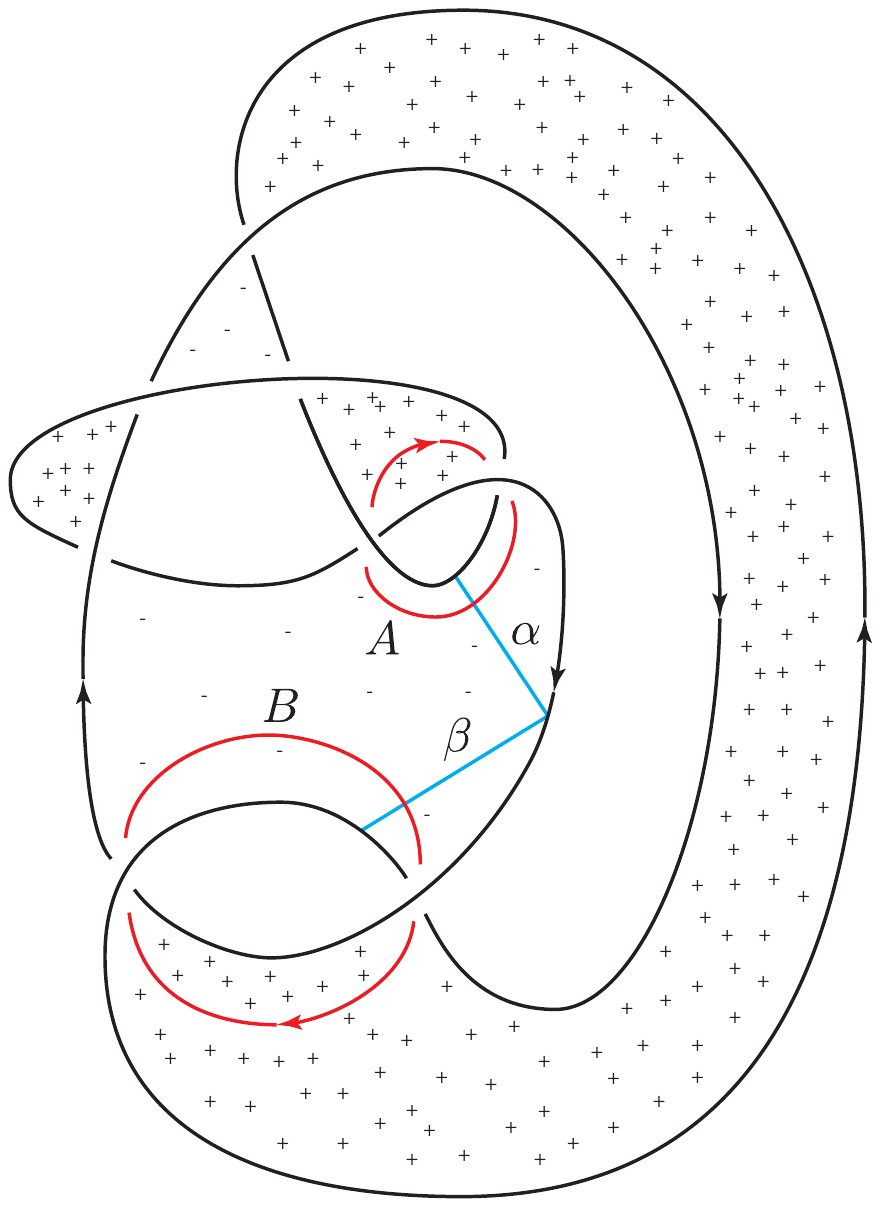} 
	\caption{} 
	\label{8_20}
\end{figure}

\begin{prop}\label{rightveeringi}  Suppose $h$ is right-veering at $C$. Then either 
\begin{itemize}
\item $c(h)\notin {\mathbb Z}$ and $i_h(\alpha)=\lfloor{c(h)}\rfloor$ or
\item $c(h)\in {\mathbb Z}$ and $i_h(\alpha)\in\{ c(h)-1,c(h)\}$.
\end{itemize}
\end{prop}

\begin{proof}
Consider first the case that $S$ is an annulus. In this case, necessarily $h=T^n$, with $n=c(h)$ and  $T$  a positive Dehn twist about the core of the annulus. Also, $i_h(\alpha)=c(h)-1$ and so the result holds. So we may assume $S$ is not an annulus. 

Let $\phi$ be the Thurston representative of $h$.  Let $\phi_C$ denote the restriction of $\phi$ to $C$. 

If $\phi$ is pseudo-Anosov or $\phi$ is reducible with $\phi_0$ pseudo-Anosov, then let $\{ x_0,x_1,...x_{q-1}\}$ be any perodic orbit of $\phi_C$.
We choose $\{ x_0,x_1,...x_{q-1} \}$ to be ordered cyclically about $C$, with order compatible with the induced orientation on $C$.

If $\phi$ is periodic, then choose a properly embedded geodesic arc $\rho$ in $S$ with both endpoints on $C$.
If $\phi$ is reducible with $\phi_0$ periodic, then since  $S_0$ is not an annulus, we may choose a properly embedded geodesic arc $\rho$ in $S_0$ with both endpoints on $C$.    In each case, 
set $x_0=\overline{\rho}(0)$ and let 
$\{ x_0,x_1,...x_{q-1}\}$ denote the corresponding periodic orbit of  $\phi_C$. Recall when constructing this orbit that $\phi$ and  $\phi_0$ respectively are isometries.  Again, we assume that $\{ x_0,x_1,...x_{q-1} \}$ is ordered cyclically about $C$, with order compatible with the induced orientation on $C$.

Notice that $$\phi_C([x_0,x_1) )= [x_r, x_{(r + 1)\,  mod\,  q})$$
for some $r, 0\le r\le q-1$.  Also, $$r=0 \Longleftrightarrow c(h)\in {\mathbb Z}\, .$$
Moreover, if $\beta$ is a properly embedded geodesic arc  orthogonal to the boundary with $\beta(0)=\overline{\beta}(0)\in [x_0,x_1)$, then $\phi_C(\beta)(0)\in [x_r, x_{(r + 1)\,  mod\,  q})$.
When $\phi$ is pseudo-Anosov or reducible with  $\phi_0$ pseudo-Anosov, this is immediate since any geodesic $\beta$ intersects the stable and unstable laminations transversally and minimally.  When $\phi$ or  $\phi_0$  is periodic, this follows from the choice of $x_0=\overline{\rho}(0)$.  Since $\overline{\rho}$ and all of its images under $\phi$ are geodesics, necessarily $\overline{\rho}$ will intersect $\beta$ transversally and minimally.

Now consider the shearing homeomorphism, $f$, that extends $\phi$ across $C\times I$, and, in particular,  maps the arc $x_0\times I$ to an arc connecting $x_0\times \{0\}$ to 
$x_{p\, mod \, q}\times \{1\}$. Define $g:C\to C$ by $g(z)=z^{\prime}$, where $f$ maps the arc $z\times I$ to an arc connecting $z\times\{0\}$ to $z^{\prime}\times \{1\}$.
Lift $g$ to a map $\tilde{g}:{\mathbb R}\to {\mathbb R}$, and index the lifts of the points $x_i$ with integers in the natural way. Notice that since 
$\phi_0([x_0,x_1)) = [x_r, x_{(r + 1)\,  mod\,  q})$, $$\tilde{g}([x_0,x_1)) = [x_{r+nq}, x_{r+1+nq})$$ where $n = \lfloor{c(h)}\rfloor$. 

Finally, we compute $i_h(\alpha)$. Let $z=\overline{\alpha}(0)\in C$. We may assume that $z\in [x_0,x_1)$, since we may relabel the indices of the $x_i$ as necessary while preserving their cyclic order.
If $r\ne 0$,  then no lift of $z$ lies in $[x_{r+nq}, x_{r+1+nq})$ and it follows that $i_h(\alpha)=\lfloor{c(h)}\rfloor$. If $r=0$, then both $\phi_C(z)$ and $z$ have unique lifts in $[x_{nq},x_{nq+1})$. Label these lifts $\widetilde{\phi_C(z)}$ and $\widetilde{z}$ respectively. If
$\widetilde{\phi_C(z)}$ lies to the right of $\widetilde{z}$ in $[x_{nq},x_{nq+1})$, then $i_h(\alpha) = c(h)$. Otherwise, $i_h(\alpha) = c(h)-1$.
\end{proof}

\begin{cor}If $h$ is right-veering at $C$, then $i_h(\alpha) \le c(h) \le i_h(\alpha) + 1$.
\end{cor}

\begin{cor}\label{compute} If $h$ is right-veering at $C$ and there exist arcs $\alpha$ and $\beta$ with $i_h(\alpha)<i_h(\beta)$, then $c(h)=i_h(\beta)=i_h(\alpha)+1$.
\end{cor}

For completeness, we record the left-veering version of these results.
 
\begin{prop} \label{leftveeringi}  Suppose $h$ is left-veering at $C$. Then either 
\begin{itemize}
\item $c(h)\notin {\mathbb Z}$ and $i_h(\alpha)=\lceil{c(h)}\rceil$ or
\item $c(h)\in {\mathbb Z}$ and $i_h(\alpha)\in\{ c(h),c(h)+1\}$.
\end{itemize}
\end{prop}

\begin{cor}If $h$ is left-veering at $C$, $i_h(\alpha) - 1 \le c(h) \le i_h(\alpha)$. \qed
\end{cor}

\begin{prop} If $h$ is neither right nor left-veering at $C$, that is $c(h)=0$, then for any $\alpha$, $i_h(\alpha)=0$.
\end{prop}

\begin{proof} If $h(\alpha)=\alpha$, then $i_h(\alpha) = 0$. So we may assume $h(\alpha)\ne\alpha$.

 Let $g=T_Ch$ where $T_C$ is a right Dehn twist about $C$.  By Proposition~\ref{formula}, $c(g)=c(h)+1=1$, thus $g$ is right-veering and by Proposition~\ref{rightveeringi}, $i_g(\alpha)\le c(g)=1$.  Since $h(\alpha)\ne\alpha$, it follows that $i_h(\alpha)=i_g(\alpha)-1\le 0$.

Next let $g'=T_C^{-1}h$.  Then $c(g')=c(h)-1=-1$, thus $g'$ is left-veering and by Proposition~\ref{leftveeringi}, $-1=c(g') \le i_{g'}(\alpha)$.  It follows that $i_h(\alpha)=i_g'(\alpha)+1\ge 0$.
\end{proof}

\begin{definition}\label{stabilize} Let $g\in Aut(S^{\prime},\partial S^{\prime})$ and let $a$ be a properly embedded arc in $S^{\prime}$.  Define $S$ to be the union of $S^{\prime}$ and a band attached to $\partial S^{\prime}$ along neighborhoods of the endpoints of $a$ in $S^{\prime}$ so that the resulting surface is orientable.   This presents $S$ as the union of $S^{\prime}$ and an annulus that intersect in a regular neighborhood of $a$.  Let $T$ be a positive Dehn twist on the annulus.  Let $\hat g$ and $\hat T$ be the extensions of $g$ and $T$, respectively, to the portion of $S$ where they are not already defined, by the identity map.  The {\sl positive stabilization of $g$ along $a$} is the homeomorphism $h=\hat T \circ \hat g :S \to S$. Refer to $a$ as the {\it plumbing arc} of the stabilization.
\end{definition}

\begin{thm}\label{stabilizedc(h)} If $h$  is a positive stabilization, and $S$ has connected boundary, then $0 \le c(h) \le 1/2$.
\end{thm}

\begin{proof} Let $w$ be the cocore of the stabilizing handle added to $S^{\prime}$ as shown in Figure~\ref{news}.  Consider the arcs $h(w)$ and $\beta=h^{-1}(w)$.  To show that $c(h)\le 1/2$, it is enough, by Proposition~\ref{formula},  to show that $c(h^2) \le 1$.  This in turn will follow, by  Proposition~\ref{rightveeringi}, if we show $i_{h^2}(\beta)=0$.

To see this, we follow the arc $\beta$ from its initial point, keep track of the edges labelled $n, e, w$ and $s$ that it crosses, and argue that if it eventually intersects $h^2(\beta)$, the intersection point will not contribute to $i_{h^2}(\beta)$.  The first edge that $\beta$ crosses is $e$.  From this point on, it never crosses either $w$ or $e$.  Thus if it intersects $h^2(\beta)$, it must do so in the portion $f$ of $\beta$ in the square bounded by $n, e, w$, and $s$.  If after hitting $e$, $\beta$ intersects $f$, the initial segments of $\beta$ and $f$ are not boundary parallel since their union misses $n$.  In this case, this point of $\beta \cap f$ contributes nothing to $i_{h^2}(\beta)$, subsequent intersections can not contribute either, and we conclude $i_{h^2}(\beta)=0$.

Now consider the case instead that $\beta$ first hits $e$, then $s$, and then $s$ again.  This forces an intersection of $\beta \cap f$, and for the same reason, we conclude $i_{h^2}(\beta)=0$.

This leaves the case that $\beta$ intersects $e, s$, and $n$ in order.  At this point, either the entire arc $\beta$ misses $f$, in which case there is nothing to prove, or $\beta$ intersects $f$.  In the latter case, the union of the initial segments of $\beta$ and $f$ are a non-separating curve dual to $n$, and hence not boundary parallel.
\end{proof}

\begin{figure}[ht]
  	\centering
   	\includegraphics[width=4truein]{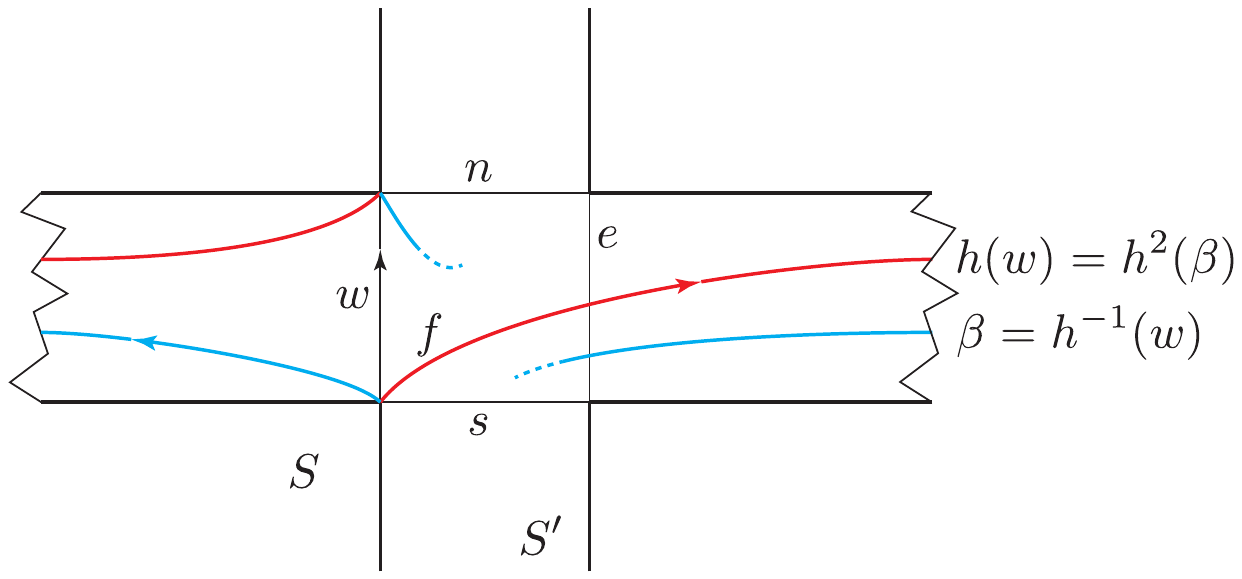} 
	\caption{} 
	\label{news}
\end{figure}

\begin{example}\label{lensexample} Let $A$ be an annulus, and let $T$ be a positive Dehn twist of $A$.  Let $h'=T^n$ with $n\ge 5$.  Next let $\sigma\subset A$ be an arc that cuts $A$ into a disk, and let $h:S \to S$ be the positive stabilization of $h'$ along $\sigma$.  Since $S$ is a punctured torus, a straightforward homology computation shows that $h$ is pseudo-Anosov.  Since $h$ is a product of positive Dehn twists, it is right-veering, hence $c(h)> 0$.  By Theorem~\ref{stabilizedc(h)}, $c(h) \le 1/2$.  Since $\chi (S)=-1$ the stable lamination can have only $1$ or $2$ cusps surrounding $\partial S$.  From this it follows that $c(h)$ is a multiple of 1/2.  It now follows that $c(h)=1/2$.
\end{example}

\begin{prop}\label{prop18} Let $h$ be a homeomorphism of a surface $S$ with connected boundary.  Suppose that $h$ is a positive stabilization of a homeomorphism $g$ of $S'$ along a plumbing arc $a$.  If $c(h)=1/2$, then $i_g(a)\ge 1$ and $i_g(a^{-1})\ge 1$. In particular, $c(g)\ge 1$ on each boundary component of  $S^{\prime}$.
\end{prop}

\begin{proof} 
By symmetry, it suffices to show that $i_g(a)\ge 1$.

Let $n, e, w, s$ be the properly embedded arcs introduced in the proof of Theorem~\ref{stabilizedc(h)} and let $\sigma=h(w)$.
Consider their geodesic representatives $\overline{n}, \overline{e} = \overline{w},\overline{s},$ and $\overline{\sigma}$. 

Notice that $\overline{n}=\overline{s}$ if and only if $S'$ is an annulus and in this case, $c(h)=1/2$ if and only if $g=T_C^n$ for some 
$n\ge 5$. So we may restrict attention to the case that $S'$ is not an annulus and $\overline{n}\cap \overline{s}=\emptyset$.

Focus attention on $\overline{n}\cap h(\overline{s})$. If this intersection is nonempty, follow the arc $h(\overline{s})$ from its initial point 
$\overline{s}(0)$ to its first intersection $h(\overline{s})(t_0)=\overline{n}(t_1)$ with $\overline{n}$. Notice that  if the geodesic representatives $\overline{n}$ and $\overline{s}$ had been chosen in $S'$, rather than in $S$, they would be equal  and isotopic to $a$.  It follows that $i_g(a)\ge 1$ if and only if $\overline{n}\cap h(\overline{s})\ne\emptyset$ with $h(\overline{s})[0,t_0] * \overline{n}[0,t_1]$ boundary parallel in $S^{\prime}$. See Figure~\ref{twisting>1}.

\begin{figure}[ht]
  	\centering
   	\includegraphics[width=3truein]{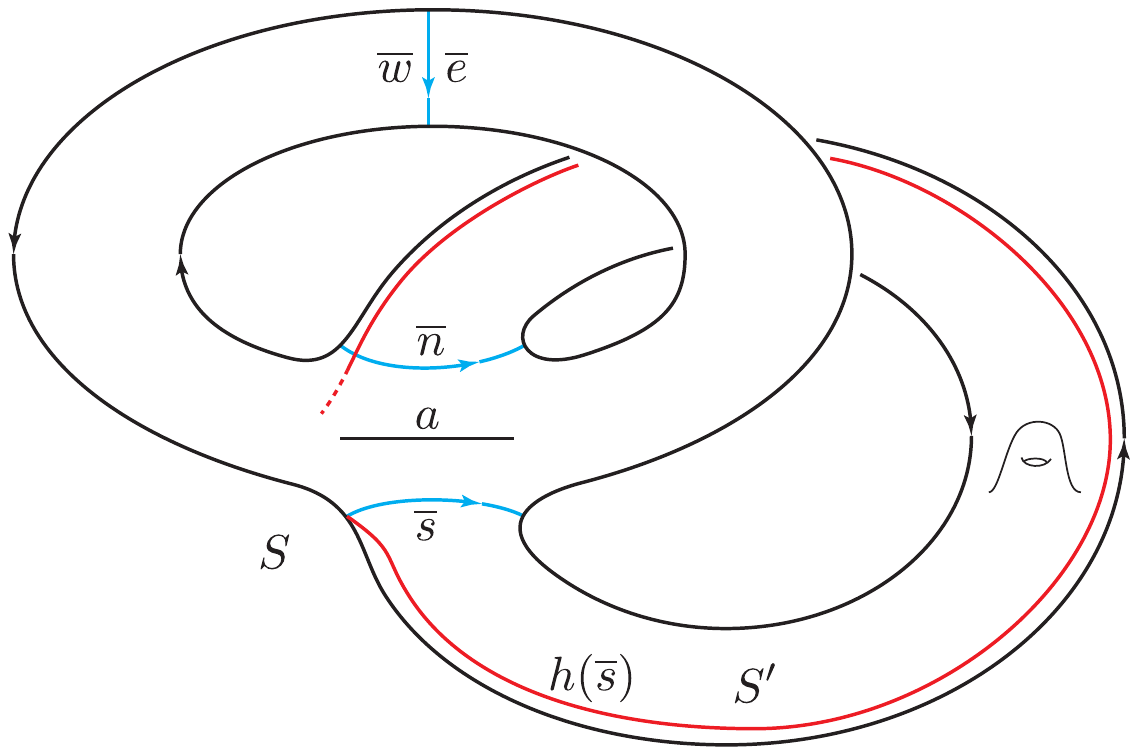} 
	\caption{} 
	\label{twisting>1}
\end{figure}

We next introduce reference points about the boundary $\partial S$. These are represented in Figure~\ref{dynamics}. As usual, let $\phi$ denote the Thurston representative of $h$. 

Consider first the case that $\phi$ is periodic or reducible with $\phi_0$ periodic. In this case, set $n=1$ and choose distinct points $x_1$ and $y_1$ so that 
$\overline{w}(0)\in (x_1,y_1)$ and also
the listing
$\{x_1,y_1,x_2,y_2\}$ is ordered cyclically and compatibly with the induced orientation of $C$, where 
$\{x_1,x_2\}$ and $\{y_1,y_2\}$ are each (order two) $\phi$-orbits  if $\phi$ is periodic and (order two) $\phi_0$ orbits if $\phi$ is reducible with $\phi_0$ periodic. 

Consider next the case that $\phi$ is pseudo-Anosov or reducible with $\phi_0$ pseudo-Anosov. 
Since $c(h)=1/2$, there are an even number, $2n$ say, of attracting periodic points and an even number, again $2n$, of repelling critical points in $C$ with respect to $\phi$ or $\phi_0$ respectively. Let $\{x_1,y_1,...x_{2n},y_{2n}\}$ be a listing of these, labelled cyclically and ordered compatibly with the
induced orientation of $C$. Cut $C$ open along these periodic points to obtain $4n$ pairwise disjoint intervals
$$(x_1,y_1), (y_1,x_2),...,(x_{2n},y_{2n}),(y_{2n},x_1)$$ 
We may further assume that the periodic points are labelled so that $\overline{w}(0)\in (x_1,y_1)$. Since $c(h)=1/2$ it
follows that $\overline{\sigma}(0)\in (x_{n+1},y_{n+1})$. 

The cyclic ordering in $C$ of endpoints of arcs in $S$ agrees with the cyclic ordering of their geodesic representatives, provided the arcs have no boundary parallel intersections.  Using this we see that since the endpoints $\partial w\cup\partial\sigma$ are cyclically ordered about $C$ as $\{w(0),\sigma(0),w(1),\sigma(1)\}$, necessarily $\overline{w}(1)\in (x_{n+1},y_{n+1})$ and $\overline{\sigma}(1)\in (x_1,y_1)$. See Figure~\ref{dynamics}.

\begin{figure}[ht]
  	\centering
   	\includegraphics[width=5truein]{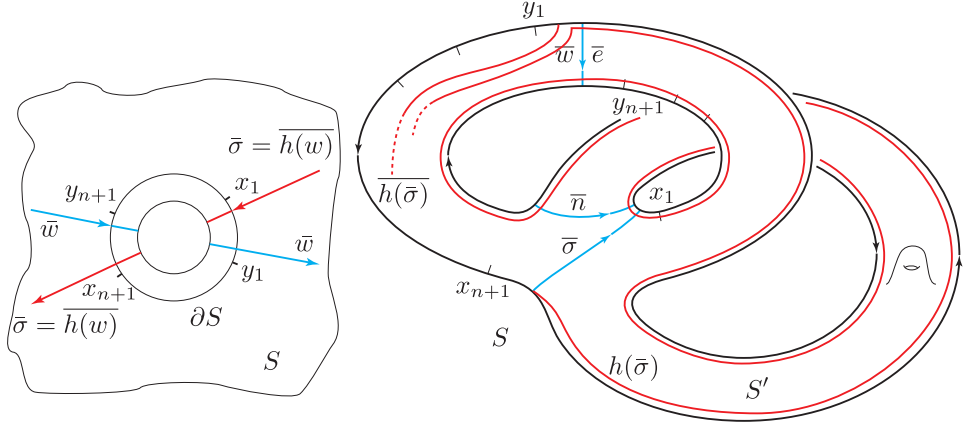} 
	\caption{Two views of $\partial S$} 
	\label{dynamics}
\end{figure}

Since $c(h)=1/2$, there is a properly embedded arc $\beta$ so that $i_{h^2}(\beta)=1$. In fact, either 
$i_{h^2}(w)=1$ or we can choose $\beta$ so that $\overline{\beta}(0)\in (y_1,x_2)$. 
If $i_{h^2}(w)=1$, then $\overline{h(\sigma)}(0)\in (\overline{w}(0),y_1)$.  In the second case, $\overline{h(\beta)}(0)\in (y_{n+1},x_{n+2})\subset  (\overline{w}(1),\overline{\sigma}(1))$.

The goal now is to relate the initial point of a geodesic orthogonal to $\partial S$, such as $\overline{h(\sigma)}$, to the homeomorphisms $\hat{T}_C$ and $\hat{g}$. Given a geodesic orthogonal to the boundary, for instance $\bar\sigma$, denote by $\bar g\bar\sigma(0)$ the initial point of $\overline{g(\bar\sigma)}$. This notation is not meant to suggest that $\bar g$ or $\bar T$ are functions that can be applied to arbitrary points of $\partial S'$.

We now show the locations of $\overline{h(\sigma)}(0)$ or $\overline{h(\beta)}(0)$ are not possible if $i_g(a) < 1$.  Recall that $h=\hat {T}_C \circ \hat g :S \to S$. We begin with $\bar{\sigma}(0), \bar{\beta}(0)\in (\bar{w}(0),\bar{s}(0))$.
So $$\bar{g}\bar{\sigma}(0)\in (\bar{g}\bar{w}(0),\bar{g}\bar{s}(0)) = (\bar{w}(0),\bar{g}\bar{s}(0)) \, .$$
If $i_g(a) < 1$ we may therefore conclude that 
$$\bar{g}\bar{\sigma}(0)\in  (\bar{w}(0),\bar{n}(0))\, .$$
and with no boundary parallel intersection. So 
$$\overline{h(\sigma)}(0)= \bar{T}_C \bar{g}\bar{\sigma}(0)\in  (\bar{T}_C \bar{w}(0),\bar{T}_C \bar{n}(0))\, .$$
and with no boundary parallel intersection. But then it follows that
$\overline{h(\sigma)}(0)\in (\bar{\sigma}(0),\bar{n}(0))$, contradicting the conclusion of the preceding paragraph. We may also conclude that $\overline{h(\beta)}(0)\in (\bar{w}(0),\bar{n}(0))$, again an impossibility. We conclude that $i_g(a)\ge 1$.
\end{proof}

\section{Stallings' twists applied to (2,1)-cables.}\label{pA}

Let $S$ be a surface properly embedded in a 3-manifold $M$, and let $C$ be an unknot which is properly embedded in $S$.  We say that $C$ is {\it untwisted relative to $S$} if its regular neighborhood in $S$ is an untwisted annulus. This means, if $D$ is� an embedded disk  in $M$  bounded by $C$, then $C$ is untwisted relative to $S$ if and only if $D$ can be isotoped relative to its
boundary so that $D$ and $S$ are transverse along $C$. A {\it Stallings' twist} \cite{St}  is a surgery along any such unknotted untwisted $C$.

Stallings' twists will be applied to $(2,1)$-cables and exploited in Section~\ref{applications}.  As a first step we show in this section that the monodromies associated to the construction are pseudo-Anosov and have fractional Dehn twist coefficient $1/2$.

Consider an open book $S^3=(F,h)$, with connected binding $k$, such that $h$ is isotopic to a pseudo-Anosov map $\psi$ and $c(h)=0$. Let $N$ denote a regular neighborhood of $k$. Abuse notation and let $F$ also denote the compact fibre of the surface bundle $S^3\setminus int N$, and  consider two copies of this compact fibre: $F_0 = F \times \{0\}$ and $F_{1/2} = F\times \{1/2\}$. Choose a hyperbolic metric with geodesic boundary for $F$ and let $F_0$ and $F_{1/2}$ inherit this metric.

Let $N_0 \subset int N$ be a second smaller regular neighborhood of $k$ and let $K$ denote the $(2,1)$-cable of $k$, embedded in $\partial N_0$ (see \cite{Rol}, p. 112).  Letting $F:S^1\times D^2 \to S^3$ be an embedding with image $N_0$, the image under $F$ of the $(2,1)$ torus knot (shown in Figure~\ref{(2,1)cable})  is what is standardly known as the $(2,1)$-cable of $k$. 

\begin{figure}[ht]
  	\centering
   	\includegraphics[width=3truein]{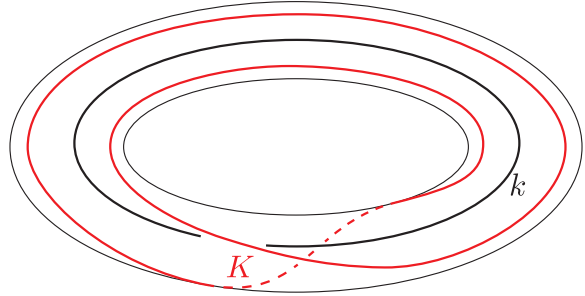} 
	\caption{The $(2,1)$-cable of $k$.} 
	\label{(2,1)cable}
\end{figure}

Let $P$ denote a pair of pants properly embedded in $N\setminus int N_1$, with boundary components $\partial F_0\cup\partial F_{1/2}$ in $\partial N$ and a third boundary component lying on $\partial N_1$, where  $N_1$ is a small regular neighborhood of $K$ lying in the interior of $N$. See Figure~\ref{pants}. 

\begin{figure}[ht]
  	\centering
   	\includegraphics[width=3truein]{pants.pdf} 
	\caption{}
	\label{pants}
\end{figure}

Set $S = F_0\cup P\cup F_{1/2}$. 
Extend the hyperbolic metric on $F_0\cup F_{1/2}$ over $P$ to obtain a hyperbolic metric  on 
$S$ with geodesic boundary.  Notice that $S$ can be thought of as the union of two copies of $F$ connected by a 1-handle. Let $\alpha = \hat{\alpha}$ denote the geodesic representative orthogonal to $\partial S$ of  a cocore of this 1-handle. We may choose this  hyperbolic metric on $S$  so that there is an orientation preserving isometry, $r:P\to P$ say, which interchanges the two components of $P \setminus \alpha$.  

Lemma~\ref{H_0formula} gives two slightly different formulas for the monodromy map $S$.  Depending on the context, it is enough to describe a representative of the free isotopy class, in which case we use $H_0$ and $r$.  To compute fractional Dehn twist coefficients we need to describe a representative of the relative isotopy class.  This is done by adjusting $r$ so that it fixes $\partial S$ pointwise.  To accomplish this, add a collar neighborhood $\partial S \times I$ to $S$, and let $R$ be the homeomorphism of $P \cup  (\partial S \times I)$ that extends $r$ and is a shear to the right by half a rotation on $\partial S \times I$.  For simplicity of notation, we  incorporate the collar into $S$ and think of $R$ as a homeomorphism of $S$.

\begin{lemma}\label{H_0formula} The surface $S$ is a fibre of $K$ in the fibred 3-manifold $S^3\setminus N_1$. Representatives of the free and relative isotopy classes of the monodromy classes are $H_0$ and $H'_0:S \to S$, respectively, which are given by

$$ H_0(z)  = \left\{ \begin{array}{ll} (x,1/2) & \mbox{ if } z=(x,0) \mbox{ lies in } F_0 \\
                           (h(x),0) & \mbox{ if } z = (x,1/2) \mbox{ lies in } F_{1/2} \\
                           (T_{\partial F_0} T_{\partial F_{1/2}})^{-1} r (z) & \mbox{ if } z \mbox{ lies in } P \end{array} \right.$$
                           
$$ H'_0(z)  = \left\{ \begin{array}{ll} H_0(z) & \mbox{ if } z \notin P \\
                           (T_{\partial F_0} T_{\partial F_{1/2}})^{-1} R(z) & \mbox{ if } z \in P \end{array} \right.$$            

\end{lemma}

\begin{proof} $(S^3\setminus int N)\setminus S$ inherits a product structure from the product structure on $(S^3\setminus int N)\setminus F$. So it suffices to understand a compatible product structure on $(N\setminus int N_1)\setminus P$. 


A meridional flow line for $S^3 \setminus int N$ is shown on the left of Figure~\ref{Productdisk}.  This splits into two subarcs that are flow lines for the $(2,1)$-cable of $k$.  For each subarc, a parallel flowline is shown as part of the boundary of a disk.  The disks in turn are unions of flowlines in $(N\setminus int N_1)\setminus P$. 

Figure~\ref{Monodromy} shows $P$ flattened out, and it records the positions of two arcs under the monodromy map which is the identity on the boundary.  The monodromy, up to isotopies fixing $\partial S$, is determined by the action on these arcs.  Direct computation shows that $(T_{\partial F_0} T_{\partial F_{1/2}})^{-1} R(z)$ agrees with the monodromy map on these arcs.
\end{proof}

\begin{figure}[ht]
  	\centering
   	\includegraphics[width=3truein]{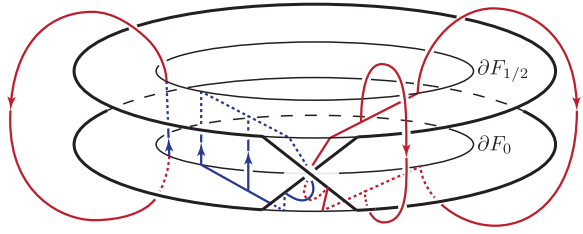} 
	\caption{Two product disks and some representative flowlines for $H'_0:S \to S$ are shown. }
	\label{Productdisk}
\end{figure}

\begin{figure}[ht]
  	\centering
   	\includegraphics[width=2.5truein]{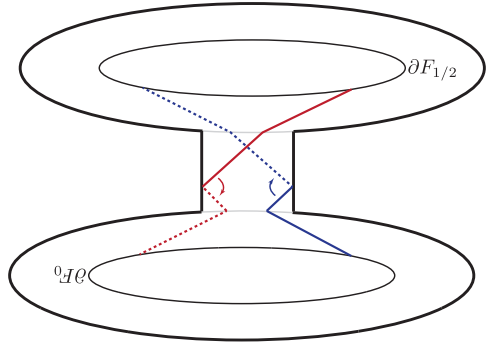} 
	\caption{Before and after images of two arcs in $P$ under the flow.}
	\label{Monodromy}
\end{figure}

Choose a simple closed curve $C$ in $S$ with two properties.  First, it intersects $F_0\cup F_{1/2}$  in  two essential arcs, $C_0 = C\cap F_0$ and $C_{1/2}=C\cap F_{1/2}$. Second, $C_i$ is nonseparating in $F_i$ for each $i = 0, 1/2$.  Now set $H= T_C\circ H_0$.

\begin{thm}\label{H} Let $H$ be the map obtained from the above construction applied to a pseudo-Anosov homeomorphism $h$ of $F$ with $c(h)=0$. The Thurston representative $\Psi$ of $H$ is pseudo-Anosov.
\end{thm}

\begin{proof} To simplify the exposition, it helps to pass to a closed hyperbolic surface $\hat{S}$ naturally associated to $S$. Topologically, $\hat{S}$ is defined by setting $$\hat{S}=S\cup_{\partial} D^2,$$
the surface obtained from $S$ by capping off $\partial S$ with a disk. Notice that $$\hat{S} = F_0\cup \hat{P}\cup F_{1/2}\, ,$$ where $\hat{P}=P\cup_{\partial} D^2$ is an annulus.  
The homeomorphism $H_0$ extends to a homeomorphism $\hat{H}_0:\hat{S}\to\hat{S}$ given by
$$ \hat{H}_0(z)  = \left\{ \begin{array}{ll} (x,1/2) & \mbox{ if } z=(x,0) \mbox{ lies in } F_0 \\
                           (h(x),0) & \mbox{ if } z = (x,1/2) \mbox{ lies in } F_{1/2} \\
                           (T_{\partial F_0} T_{\partial F_{1/2}})^{-1} \hat{r}(z) & \mbox{ if } z \mbox{ lies in } \hat{P} \end{array} \right.$$
where $\hat{r}:\hat{P}\to\hat{P}$ is a homeomorphism, uniquely determined up to isotopy, extending $r:P\to P$.

Now consider the closed hyperbolic surface $F_0\cup_f F_{1/2}$, obtained by gluing the hyperbolic surfaces $F_0$ and $F_{1/2}$ by the isometry $f:\partial F_0\to \partial  F_1$ given by $f(x,0) = (x,1/2)$.  Choose a  diffeomorphism from $\hat{S}$  to $F_0\cup_f  F_{1/2}$, and use this diffeomorphism to pull back the hyperbolic metric on $F_0\cup_f F_{1/2}$ to a hyperbolic metric on $\hat{S}$. Let $A$ denote the geodesic  simple closed curve which is the pull back to $\hat{S}$  of $\partial F_0 = \partial F_{1/2}\subset F_0\cup_f F_{1/2}$. Notice that 
 $$ \hat{H}_0(z)  = \left\{ \begin{array}{ll} (x,1/2) & \mbox{ if } z=(x,0) \mbox{ lies in } F_0 \\
                           (h(x),0) & \mbox{ if } z = (x,1/2) \mbox{ lies in } F_{1/2} \\
                          T_A^{-2} \,\, \hat{r}(z) & \mbox{ if } z \mbox{ lies in } \hat{P} \end{array} \right.$$

Set $$\hat{H}=T_C \circ \hat{H}_0\, .$$

We will use the following characterization of periodic and reducible maps (\cite{Th, CB}). 

\begin{lemma}\label{first} A homeomorphism $h$ of a surface $S$ has a periodic or reducible representative if and only if there exists a possibly immersed multi-curve $\gamma = \rho_1 \cup \dots \cup \rho_n$ where each $\rho_i$ is an essential, embedded, curve in $S$, and, up to isotopy, $\gamma$ is fixed by $h$.  In the periodic case, each $\rho_i$ may be assumed to be non-separating, and in the reducible case, $\gamma$ may be chosen to be embedded. \qed
\end{lemma}

\begin{cor} If the Thurston representative $\hat\Psi$ of $\hat H$ is pseudo-Anosov, then the Thurston representative $\Psi$ of $H$ is also pseudo-Anosov.
\end{cor}

\begin{proof} By contradiction, suppose that $H$ is either periodic or reducible, and let $\gamma$ be the multi-curve guaranteed by Lemma~\ref{first}.  Since the curves $\gamma$, $H_0(\gamma)$, and $T_C^{-1}(\gamma)$ and the image of the isotopy between $H_0(\gamma)$ and $T_C^{-1}(\gamma)$ all lie in $S$ and therefore in  $\hat{S}$, it follows that $H(\gamma)$ is isotopic to $\gamma$.  Furthermore, no component of $\gamma$  is  boundary parallel in $S$.  Thus applying Lemma~\ref{first} to $\gamma \subset \hat S$ shows that $\hat H$ is either periodic or reducible.
\end{proof}

Theorem~\ref{H} therefore follows from the following theorem. \end{proof}

\begin{thm}\label{hatH}  The Thurston representative $\hat{\Psi}$ of $\hat{H}$ is pseudo-Anosov.
\end{thm}

\begin{proof} 

 Suppose  instead that $\hat{\Psi}$ is periodic or reducible. Choose $\gamma$ as guaranteed by Lemma~\ref{first} so that it is either an embedded curve or is a union of non-separating curves.  Necessarily $\hat{H}_0(\gamma)$ is isotopic to $T^{-1}_C(\gamma)$.  We next establish notation for a sequence of lemmas that will be used to derive a contradiction in Corollary~\ref{last}. 

Let $\lambda^s$ and $\lambda^u$ denote the stable and unstable geodesic laminations of $\psi$. 
Let $L^s_1, ..., L^s_n$ , $n\ge 1$, denote the leaves of $\lambda^s$ bounding the complementary region of $\lambda^s$ containing $\partial F$, ordered cyclically about $\partial F$.
Similarly, let $L^u_1, ..., L^u_n$ denote the leaves of $\lambda^u$ bounding the complementary region of $\lambda^u$ containing $\partial F$, ordered cyclically about $\partial F$.

The unstable prongs, i.e., the geodesics perpendicular to $\partial F$ found between each consecutive pair 
$L^u_i, L^u_{i+1}$, canonically cut each $L_i^s$ into two open intervals $L_{i-}^s$ and $L_{i+}^s$. Similarly, the stable prongs canonically cut each $L_i^u$ into two open intervals $L_{i-}^u$ and $L_{i+}^u$. Since $c(h)=0$, we may assume that both $\psi$ and $h$ fix the leaves $L_i^s, L_i^u, L_{i\pm}^s$, and $L_{i\pm}^u$, for $1\le i\le n$. $L_i^s$, the leaves $L_i^u$, the leaves $L_{i\pm}^s$, and the leaves $L_{i\pm}^u$, $1\le i\le n$. 

To simplify the exposition (and the notation), we now assume that $\lambda^s\times\{0,1/2\}$, $\lambda^u\times\{0,1/2\}$, $A$, $C$, and $\gamma$ are each geodesic in $\hat{S}$. Let $\eta$ denote the geodesic representative of $\hat{H}_0(\gamma)$, and choose  an isotopy representative of  $T_C^{-1}$  so that  $T_C^{-1}(\gamma)$ is a geodesic. 

Let $W$ denote the closed complementary region of $\lambda^s\times \{0,1/2\}$ which contains $A$.
Since $\gamma$ meets $\lambda^s\times \{0,1/2\}$  efficiently, so does $\hat{H}_0(\gamma)$. It follows that we can isotope $\hat{H}_0(\gamma)$ to be geodesic while preserving this efficient intersection with $\lambda^s\times \{0,1/2\}$. The following therefore holds. 

\begin{lemma}\label{path} If a component 
 of $\hat{H}_0(\gamma)\cap W$ is a path connecting $L^s_i\times\{0\}$ and $L^s_j\times\{1/2\}$ for some $i,j$, $1\le i,j\le n$, then the corresponding component of  $\eta\cap W$ is also a path connecting $L^s_i\times\{1/2\}$ and $L^s_j\times\{0\}$.  \qed
\end{lemma}

\begin{lemma}\label{comp}
Let $O$ be a geodesic simple closed curve in $\hat{S}$, and for each integer $l$, let $O_l$ denote the geodesic representative of $T_A^l(O)$. If a component of $O\cap W$ is a path connecting $L^s_i\times\{0\}$ and $L^s_j\times\{1/2\}$ for some $i,j$, $1\le i,j\le n$, then the corresponding component of  $O_l\cap W$ is also a path connecting $L^s_i\times\{1/2\}$ and $L^s_j\times\{0\}$. 
\end{lemma}

\begin{proof}
The leaves $L_i^s\times\{0\}$ which intersect $O$ are uniquely determined by the constraint that the intersection $O\cap (L_i^s\times\{0\})$ is efficient. This efficiency of intersection is unaffected by composition with $T_A^l$.
\end{proof}

Symmetric statements hold relative to the unstable lamination $\lambda^u\times \{0,1/2\}$.

\begin{lemma} Both $\gamma\cap C\ne \emptyset$ and $\gamma\cap A\ne \emptyset$.
\end{lemma}

\begin{proof} If $\gamma\cap C=\emptyset$, then $\gamma$ is fixed by $T_C^{-1}$ and hence, up to isotopy, also by $\hat{H}_0$.  Since $\psi$ is pseudo-Anosov, $\hat{H}_0$ fixes only $A$, up to isotopy, and it follows that $\gamma\cap C \ne \emptyset$. 

It follows that if $\gamma\cap A =  \emptyset$, then necessarily  $T_C^{-1}(\gamma)\cap A \ne  \emptyset$. But we know
$\hat{H}_0(\gamma)=T^{-1}_C(\gamma)$ and hence  $$|T^{-1}_C(\gamma)\cap A| =|\hat{H}_0(\gamma)\cap A|=|\hat{H}_0(\gamma)\cap \hat{H}_0(A)| = |\gamma\cap A| =0 \, .$$ So $\gamma\cap A \ne \emptyset$
 (and $T_C^{-1}(\gamma)\cap A \ne  \emptyset$).
\end{proof}

We now use $C$ to get a measure of the amount of twisting of $\hat{H}_0(\gamma)$ and $T_C^{-1}(\gamma)$ respectively about $A$. We do this as follows. First consider the components of $\hat{S}\setminus{(\gamma\cup A\cup C)}$. Call such a complementary region a {\it triangular disk region} if it is a disk with piecewise geodesic boundary consisting of exactly three geodesic subarcs, one contained in  each of $\gamma$, $A$ and $C$ respectively.
Choose a regular neighborhood $X$ of $A$ so that 
\begin{enumerate}
\item all such triangular disk regions are contained in the interior of $X$,
\item all intersections of $\partial X$ with $C$, $\gamma$, $\hat{H}_0(\gamma)$, and $T_C^{-1}(\gamma)$ are transverse,
\item $C\cap X$ consists of two essential embedded arcs: $\sigma_1$ and $\sigma_2$,
\item $\gamma\cap X$ consists of $a$ essential embedded arcs: $\gamma_1,...,\gamma_a$, for some $a\ge 2$,
\item $\eta\cap X$ consists of $a$ essential embedded arcs: $\eta_1,...,\eta_a$,
\item $T_C^{-1}(\gamma)\cap X$ consists of $a$ essential embedded arcs: $\tau_1,...,\tau_a$, and
\item  for $1\le i\le 2$ and $1\le j\le a$, each of the arcs $\sigma_i$,  $\gamma_j$, $\eta_j$, and $\tau_j$ has nonempty (necessarily minimal) intersection with each of $\cup L_i^s\times \{0\}$, $\cup L_i^s\times\{1/2\}$,  $\cup L_i^u\times \{0\}$,  and $\cup L_i^u\times\{1/2\}$.
\end{enumerate} 
In (4)--(6) we are using the fact that $\hat{H}_0(A)=A$,  thus  $$|T_C^{-1}(\gamma)\cap A| = |\hat{H}_0(\gamma)\cap A| = |\gamma\cap A|= a\, .$$
Choose the indices $j$,  $1\le j\le a$, so that $\eta_j$ is the component corresponding to $\hat{H}_0(\gamma_j)$.

Orient the arcs $\sigma_1,\sigma_2$, $\gamma_i$, $\eta_i$, and  $\tau_i$, $1\le i \le a$, so that they run from $\partial X\cap (F\times \{0\})$ to $\partial X\cap (F\times \{1/2\})$.

\begin{lemma}\label{small} For all $i$ and $j$, $1\le i\le a$ and  $1\le j\le 2$, 
 $|\langle \tau_i,\sigma_j\rangle| \le 1$.
\end{lemma}

\begin{proof}
Since $\tau_i \subset T_C^{-1}(\gamma)$ and $\sigma_j \subset C$ each component $\tau_i$ can meet each of $\sigma_1$ and $\sigma_2$ at most once. 
\end{proof}

\begin{lemma} For some $i_0$ and $j_0$, $1\le i_0\le a$ and  $1\le j_0\le 2$,  $\langle\gamma_{i_0}, \sigma_{j_0}\rangle <0$.
\end{lemma}

\begin{proof} 
Points of intersection $\gamma\cap C$ lie either in $X$ or outside $X$. The condition that all  triangular disk regions lie in $X$ guarantees that any point of intersection $\gamma\cap C$ lying outside $X$ results in increased geometric intersection of $T_C^{-1}(\gamma)$ with $A$. And any point of intersection $\gamma\cap C$ lying inside $X$ and with positive intersection number $\langle \gamma_i,\sigma_j\rangle$  results in increased geometric intersection of $T_C^{-1}(\gamma)$ with $A$. Since $\gamma\cap C$ is nonempty, it follows that if all intersection numbers $\langle \gamma_i,\sigma_j\rangle$ are nonnegative, then necessarily $|T_C^{-1}(\gamma)\cap A| > |\gamma\cap A|$. This is impossible.
\end{proof}

Rechoose indices as necessary so that $$\langle \gamma_1,\sigma_1\rangle < 0$$ and $\eta_1$ is the component of $\eta$ corresponding to $\hat H_0(\gamma_1)$.

\begin{cor}\label{last} The geodesic representative of $\hat H_0(\gamma)$, $\eta$, and $T_C^{-1}(\gamma)$ are not equal, and therefore $\hat{H}(\gamma)$ is not isotopic to $\gamma$.
\end{cor}

\begin{proof}
Focus on $\gamma_1$ and $T_A^2(\eta_1)$, the latter arc isotoped rel endpoints to be geodesic and oriented so that it runs from $\partial X\cap (F\times \{0\})$ to $\partial X\cap (F\times \{1/2\})$. The arc $\gamma_1\cap W$ is a path from $L_i^s\times\{0\}$ to $L_j^s\times\{1/2\}$ for some $i,j,\, 1\le i,j\le n$.  Since $\hat{H}_0$ switches $F\times \{0\}$ and $F\times \{1/2\}$, and  since $c(h)=0$, it follows from Lemma~\ref{comp} that both $\eta_1\cap W$ and  $T_A^2(\eta_1)\cap W$ are paths from $L_j^s\times\{0\}$ to $L_i^s\times\{1/2\}$. There are two possibilities: either $T_A^2(\eta_1)\cap\gamma_1 = \emptyset$ or 
$T_A^2(\eta_1)\cap\gamma_1 \ne \emptyset$.

Consider first the case that  $T_A^2(\eta_1)\cap\gamma_1 = \emptyset$. In this case, 
$$\langle T_A^2(\eta_1),\sigma_1\rangle\le 0$$ and hence
$$\langle \eta_1,\sigma_1\rangle \le -2\, .$$ It follows from Lemma \ref{small} that
$\eta_1\ne \tau_m$ for all $m, 1\le m\le a$.

Next consider the case that $T_A^2(\eta_1)\cap\gamma_1 \ne \emptyset$. Recall that $\eta_1$ is the portion of the geodesic, corresponding to $\gamma_1$, that is formed by applying $\hat H_0$ to $\gamma$.  The action of $\hat H_0$ on $W$ is $T^{-2}_A$, thus $T_A^2(\eta_1)$ corresponds to acting on $\gamma_1$ by the map $\hat r$ used in the definition of $\hat H_0$.  If $i \ne j$, $\hat r$ would force $T_A^2(\eta_1)\cap\gamma_1 = \emptyset$.  It follows that $i=j$.  Moreover, $\langle \gamma_1,\sigma_1\rangle < 0$ implies $|\langle T_A^2(\eta_1),\sigma_1\rangle|\le1$.

If $\langle T_A^2(\eta_1),\sigma_1\rangle\le 0$, the argument of the first case applies.
So we may assume $\langle T_A^2(\eta_1),\sigma_1\rangle= 1$. This is possible only if
$\gamma_1$, $T_A^2\eta_1$ and $\sigma_1$ are all paths connecting $L_i^s\times\{0\}$ to $L_i^s\times\{1/2\}$, pairwise isotopic through paths connecting $L_i^s\times\{0\}$ to $L_i^s\times\{1/2\}$. In other words, up to composition by $T_A^l$ for some integer $l$, they must lie as shown in Figure~\ref{intersect}.a. The arc $\eta_1$ is added in Figure \ref{intersect}.b. \begin{figure}[ht]
  	\centering
   	\includegraphics[width=3truein]{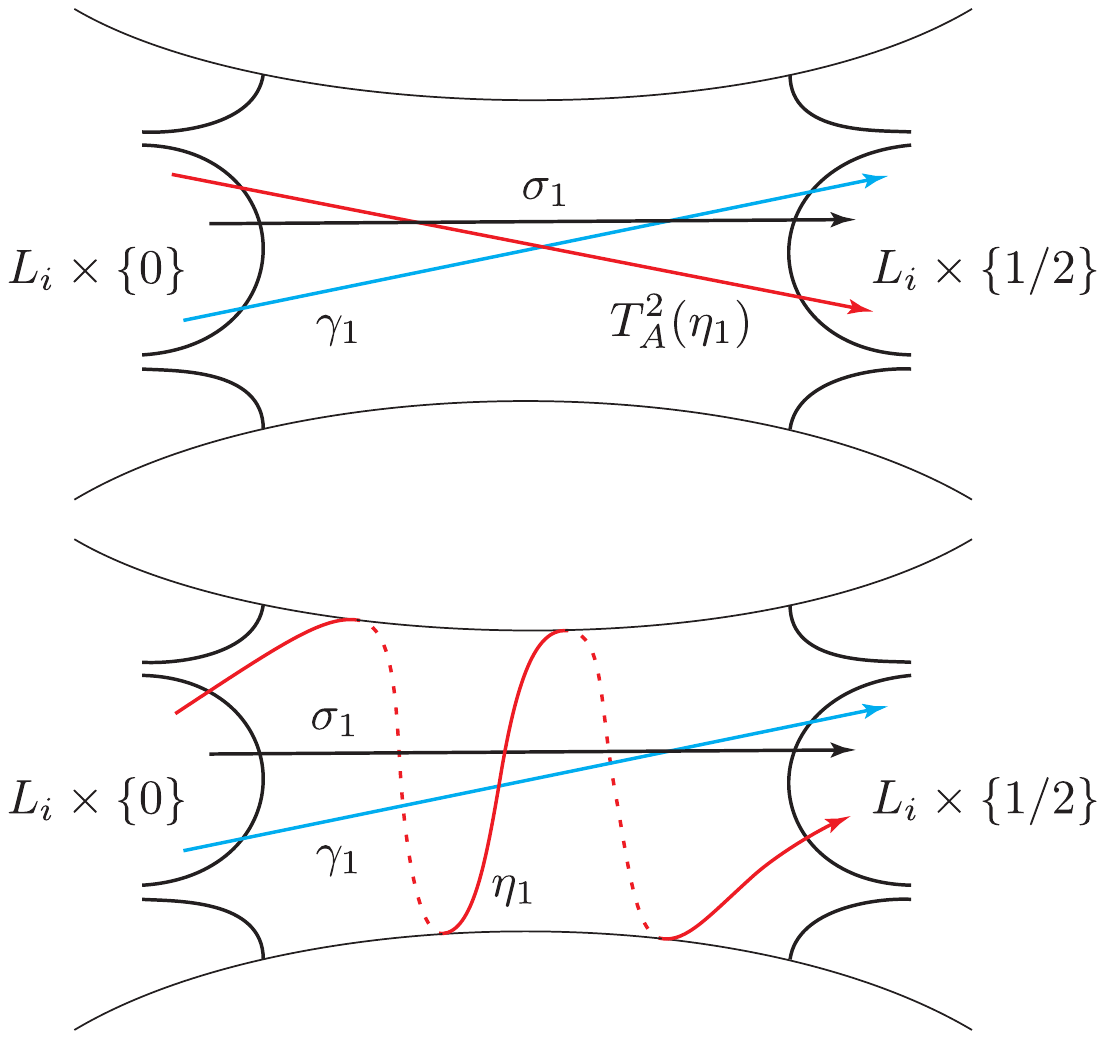} 
	\caption{}
	\label{intersect}
\end{figure}

The next step is to show there is a geodesic arc, $\beta$, properly embedded and essential in $W$, such that  
$\beta$ is disjoint from $\gamma_1\cup T_A^2(\eta_1)$. When $n>1$, choose  $\beta$ to be a geodesic  path from $L_p^s\times\{0\}$ to $L_p^s\times\{1/2\}$ for some $p\ne i$.
When $n=1$, we argue similarly but  work  instead with the leaves $L_{1\epsilon}^s$,   $\epsilon\in\{\pm\}$. The arcs $\gamma_1$, $T_A^2\eta_1$ and $\sigma_1$ are all paths connecting
$L_{1\epsilon}^s\times\{0\}$ to $L_{1\epsilon}^s\times\{1/2\}$, pairwise isotopic through paths connecting
$L_{1\epsilon}^s\times\{0\}$ to $L_{1\epsilon}^s\times\{1/2\}$. We then choose $\beta$ to be a geodesic path from $L_{1\delta}^s\times\{0\}$ to $L_{1\delta}^s\times\{1/2\}$, where $\delta = -\epsilon$.

Thus in either case, $\beta$ satisfies $|\beta\cap \eta_1|=2$, whereas $|\beta\cap\tau_m|\le 1$ for all $l, 1\le l\le a$. So again we may conclude $\eta_1\ne \tau_m$ for all $m, 1\le m\le a$.
\end{proof}

Theorem~\ref{hatH} now follows from Corollary~\ref{last}\end{proof}

\begin{prop}\label{H_0'} With $H_0'$ as defined in Lemma~\ref{H_0formula}, and $H'=T_C \circ H'_0$, $1/2 \le c(H')\le1$.
\end{prop}

\begin{proof}  Let $\alpha \subset S$ be the arc fixed by the involution $r:P \to P$.  To estimate $c(H')$, we first compute $i_{H'}(\alpha)$.  Loosely speaking, it is enough to compute just $R$ and $T_C$ on the initial portion of $\alpha$.  More precisely, $H'(\alpha)=\beta_1*\beta_2$ where $\beta_1$ is an arc ending transversely on $F\times \{0\}$, and $\beta_2$ is an arc which starts on $F\times \{0\}$ and has an essential first return to $F\times \{0\}$.  Figure~\ref{c(H')} shows $\beta_1$ which is computed by applying $R$ and $T_C$ to the initial portion of $\alpha$.  Only the first intersection of $\beta_1$ and $\alpha$ can contribute to $i_{H'}(\alpha)$, but this does not yield a boundary parallel union of initial segments.  It follows that $i_{H'}(\alpha)=0$, and thus by Proposition~\ref{rightveeringi}, $0 \le c(H') \le 1$.

\begin{figure}[ht]
  	\centering
   	\includegraphics[width=3.7truein]{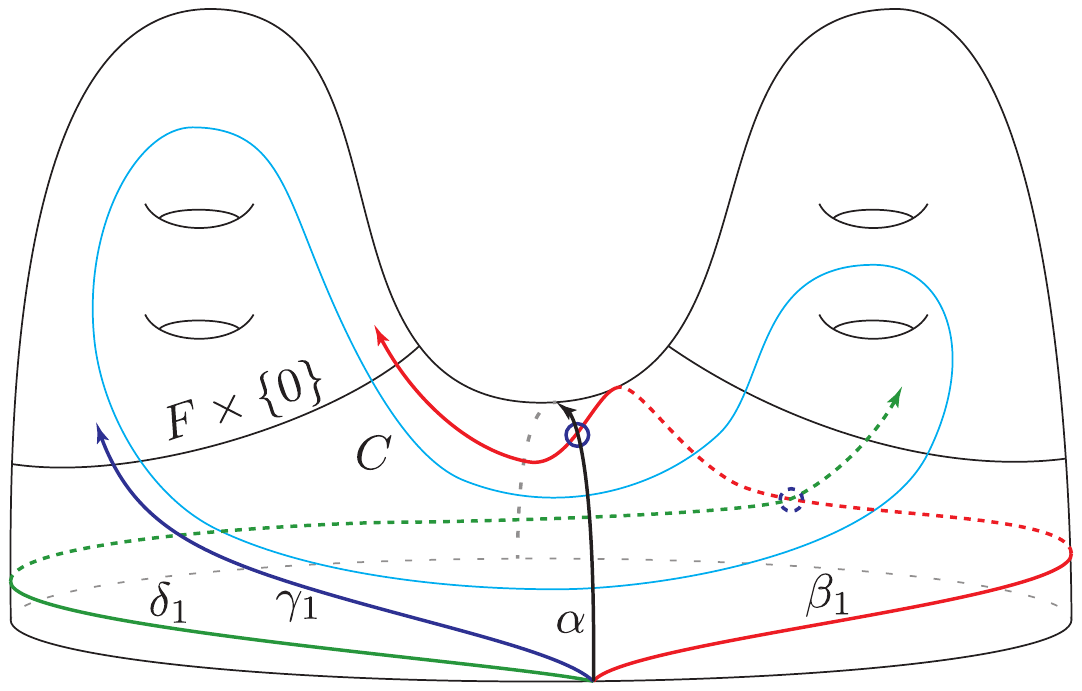} 
	\caption{}
	\label{c(H')}
\end{figure}

The next step is to estimate $c(H'^2)$ by computing $H'^{-1}(\alpha)$ and comparing it to $H'(\alpha)$.  Towards this end, write $T_C^{-1}(\alpha)=\gamma_1*\gamma_2$ where $\gamma_1$ is essentially truncated at $F\times \{0\}$ as above.  To compute $H'^{-1}(\alpha)=\delta_1*\delta_2$ where $\delta_1$ is essentially truncated at $F\times \{1/2\}$ by applying the remaining portion of $H'^{-1}$ to $\beta_1$.  The single intersection of $\gamma_1$ and $\beta_1$ is enough to conclude that $i_{H'^2}(H'^{-1}(\alpha))=1$.  Thus by Lemma~\ref{rightveeringi} we have $1\le c(H'^2)\le 2$ and it follows by Lemma~\ref{formula} that $1/2\le c(H') \le 1$.
\end{proof}

We show in the next section, Corollary~\ref{H'inS3}, that in $S^3$, $c(H')$ can only be $1/2$.

\section{Stabilization in $S^3$}\label{S3iS}
The first portion of this section is devoted to proving the result, Theorem~\ref{onlyvalues},  that the only values fractional Dehn twist coefficients take on in $S^3$ are $0$ and $1/n$ where $n$ is an integer satisfying $|n|\ge 2$.

This leads to the statement of Conjecture~\ref{main1}, or in light of Theorem~\ref{onlyvalues}, the equivalent statement, Conjecture~\ref{main2}, either of which imply that a fibred knot in $S^3$ with fractional Dehn twist coefficient $1/2$ can never be destabilized. This consequence, in the most important case when the monodromy is pseudo-Anosov, was first stated,  without proof,
 by Gabai in \cite{Ga3}. 

Part of our interest in this result is that it appears to depend on deep facts about $S^3$, most likely Gabai's thin position strategy, in an essential way.  Indeed, very simple examples, such as the examples of Proposition~\ref{lenspace}, show the result fails even in lens spaces.

We begin by considering more closely the case that the Thurston representative $\phi$ of the  monodromy $h$ is either periodic or reducible with $\phi_0$ periodic.
\begin{prop}\cite{Seifert} Let $S^3=(S,h)$ be an open book decomposition which has connected binding $k$. If the Thurston representative of $h$ is periodic, then $k$ is the unknot or  a 
$(p,q)$-torus knot for some relatively prime integers $p$ and $q$ satisfying $|p|>1$.
\end{prop}

\begin{prop} Let $S^3=(S,h)$ be an open book decomposition which has connected binding $k$. If the Thurston representative, $\phi$, of $h$ is reducible with $\phi_0$ periodic, then $k$ is a $(p,q)$-cable knot, for some relatively prime integers $p$ and $q$ satisfying $|p|>1$, and $c(h)=1/(pq)$.
\end{prop}

\begin{proof}
Suppose $\phi$ is reducible with  $\phi_0:S_0 \to S_0$ periodic. Denote the boundary components of  $S_0$ by $C=k = \partial S, C_1,\dots, C_n$.  Let $T_i$ denote the torus in $S^3$ that is the union of flow lines of $\phi_0$ intersecting $C_i$.  Since $\phi_0$ may permute the $C_i$'s, the $T_i$'s are not necessarily distinct.  Let $V_i$ denote a solid torus in $S^3$ with $\partial V_i=T_i$.

Each $T_i$ is essential in $S^3-k$.  To see this, notice that an infinite cyclic cover of $S^3-k$ can be created by gluing copies of $S^3-k$ split along $S$.  The inverse image of $T_i$ in this cover is an infinite cylinder which is incompressible since $C_i$ is essential in $S$.

It follows that every $V_i$ contains $k$.  By connectivity of $S_0$ there can be only one such solid torus, thus $V=V_1=\dots=V_n$ and $T=T_1=\dots T_n$.  Consider a closed orbit $\gamma$ of $\phi_0$ that lies on the boundary of a regular neighborhood  $N(k)$ of $k$.  If this is a meridian, then $c(h)=0$.  Otherwise the orbits of $\phi_0$, together with $k$, define a Seifert fibre space structure on $V$.  

A Seifert fibre space structure on $V$ can have at most one singular fibre.  If $k$ is a singular fibre, the leaf space of $V-N(k)$ is an annulus.  Moreover, this annulus is covered by $S_0$.  It follows that $S_0$ is itself an annulus, thereby contradicting the choice of   $S_0$ as a reducing surface for $\phi$.

It follows that $k$ is a non-singular fibre and hence a $(p,q)$-cable of the core of $V$. Let $K_V$ denote the knot which is the core of $V$, and let $X_V$ denote the complement of $V$ in $S^3$.  Coordinates $(p,q)$ are chosen so that a $(1,0)$ curve is a longitude for $K_V$, and a $(0,1)$ curve is a meridian for $K_V$.

Let $S^{\prime\prime}$ denote the subsurface $S\setminus int S_0$ of $S$. Notice that the orientation on $S_0$ induces a $\phi_0$-invariant orientation on the intersection $T\cap S=C_1\cup\dots\cup C_n$.  Equivalently, the oriented complementary surface $S^{\prime\prime}$ intersects $\partial X_V$ in  $n\ge 1$ parallel consistently oriented simple closed curves: $C_1,\dots,C_n$. Necessarily these curves are $(1,0)$ curves and $S^{\prime\prime}$ has $n$ connected components. Moreover, the complement of $S^{\prime\prime}$ in $X_V$ is an $I$-bundle.  It follows that $K_V$ is fibered and $S^{\prime\prime}$ consists of $n$ copies of a fibre $S_V$ of $K_V$. It follows also that $p=n$.  Moreover, since $S_0$ cannot be an annulus, $|p|>1$.

To compute $c(h)$, consider a periodic orbit $\gamma$ near $C$, and express it in (longitude, meridian) coordinates with respect to $C$.  The orbits of $\phi_0$ define the Seifert fibration on $V$, and since $C$ is a non-singular fibre, $\gamma$ is also a $(p,q)$ curve with respect to $K_V$.  It follows that $\gamma$ intersects a meridian of $C$ once.  To compute the number of times $\gamma$ intersects $S$, it is enough to compute the intersection between a $(p,q)$ curve and the $p$ $(1,0)$ curves that are $S\cap T$.  It follows that $\gamma$ is a $(1,pq)$ curve with respect to $C$, thus $c(h)=1/(pq)$.
\end{proof}

\begin{cor}\label{pqgames} Let $S^3=(S,h)$ be an open book decomposition which has connected binding $k$. Suppose that the Thurston representative, $\phi$, of $h$ is either periodic or reducible with $\phi_0$ periodic. Then either $k$ is the unknot and $c(h)=0$ or $c(h)=1/n$, where  $|n|\ge 2$ is the slope of the cabling annulus.

\end{cor}

\begin{proof}
Since $k$ is either a $(p,q)$-torus knot or a $(p,q)$-cable knot, it follows that either $k$ is the unknot, or the cabling annulus is incompressible and boundary incompressible in $S^3-int\,N(k)$. Since the cabling annulus has slope $pq$, it follows that either $k$ is the unknot  and $c(h)=0$ or $k$ is not the unknot and $c(h) = 1/(pq)$. 
\end{proof}

\begin{thm}\label{Gatwist}\cite{Ga3} Let $S^3=(S,h)$ be an open book decomposition which has connected binding $k$. Suppose that $\phi$  is either  pseudo-Anosov  or  reducible with $\phi_0$ pseudo-Anosov.  Then either $c(h)=0$ or $c(h)=1/r$, where $2\le |r|\le 4(genus(k))-2$.
\end{thm}

\begin{proof} Theorem~8.8 of \cite{Ga3} states that the degeneracy, $d(k)$, of the complement of $k$ has one of the forms
\begin{enumerate}
\item $r(1/0)$, with $1\le r\le 4(genus(k)) - 2$
\item $r/1$, with $2\le |r| \le 4(genus(k))-2$
\end{enumerate}
(See \cite{Ro2} for a second proof of the lower bound in (2).)
As noted in Section~\ref{basics}, when $M=S^3$, $c(h)$ is the reciprocal of the slope of $d(k)$. It follows that  $c(h)$ is either $0$ or $1/r$ for some  integer $r$, $2\le |r| \le 4(genus(k))-2$. 
\end{proof}

Combining Corollary~\ref{pqgames} and Theorem~\ref{Gatwist} proves

\begin{thm}\label{onlyvalues} If $(S,h)$ is an open book decomposition of $S^3$ with connected binding, then $c(h)$ equals $0$ or $1/n$ for some integer $n$, $|n|\ge 2$.\qed
\end{thm}

Theorem~\ref{Gatwist} and Proposition~\ref{H_0'} immediately imply 

\begin{cor}\label{H'inS3} Let $H_0'$ be the monodromy of a (2,1)-cable of a knot in $S^3$ with pseudo-Anosov monodromy and fractional Dehn twist coefficient $0$. Then $H'=T_C \circ H'_0$, as constructed above, satisfies $c(H')=1/2$.\qed
\end{cor}

Among all the possible values of the fractional Dehn twist coefficient that arise in $S^3$, it is 1/2 that plays the most mysterious role knot theory. The following was first stated by Gabai in \cite{Ga3} along with a possible proof strategy.

\begin{conjecture}\label{main1} Let $S^3=(S,h)$ be an open book decomposition which is stabilized and has connected binding $K$. Then $c(h)\ne 1/2$.
\end{conjecture}

By Theorem~\ref{onlyvalues}, this conjecture can also be stated as follows.

\begin{conjecture} \label{main2} Let $S^3=(F,h)$ be an open book decomposition which is stabilized and has connected binding. Then either $c(h)=0$ or $c(h)=1/n$, where $n\ge 3$.
\end{conjecture}
%

 An appealing aspect of these conjectures is that very simple examples show that they are false outside of $S^3$. Specifically, Example~\ref{lensexample} in Section~\ref{basics} proves

\begin{prop}\label{lenspace} Let $p\ge 5$ and $q = 1$. Then $L(p,q)$ has an open book decomposition with connected binding which is stabilized and satisfies $c(h)=1/2$.
\qed\end{prop}

\section{Applications}\label{applications}

In this section, we give an application of our results in contact topology by associating to each pair $(S,h)$ a contact structure $\xi$ via the Giroux correspondence \cite{Gi}.  This correspondence is many to one, and Giroux shows that the failure of injectivity is generated by positive stabilization (see Definition~\ref{stabilize}).  A natural goal is to try to determine properties of $\xi$ from a single pair $(S,h)$.

Some results in this direction are as follows \cite{HKM1}. If $c(h) \le 0$ then $\xi$ is overtwisted.  If $c(h) \ge 1$ then $\xi$ is tight.  Contrast this with the uncheckable theorem that $\xi$ is tight if and only if for {\it every} $(S,h)$ determining $\xi$, $c(h) >0$.  This leads to the conjecture by Honda, Kazez, and Mati\'c \cite{HKM1} that if a single representative is not destabilizable and has $c(h) >0$ (or equivalently is right-veering) then $\xi$ is tight.

Lekili~\cite{Lekili} produced counterexamples to this conjecture with $S$ a 4 times punctured sphere.  Lisca~\cite{Lisca} constructed an infinite number of counterexamples for the same surface.   Ito and Kawamuro~\cite{IK} have produced an even larger set of counterexamples on the 4 times punctured sphere. 

This should be contrasted with the work of Colin and Honda \cite{CoHo} in which they show that if an open book has connected boundary, pseudo-Anosov monodromy, and fractional Dehn twist coefficient $k/n >0$, then $k>1$ implies the associated contact structure is not only tight, but  universally tight, and the universal cover is $\mathbb R^3$.

Our examples show that it is often very easy to recognize an overtwisted contact structure using essentially no technology. It is enough, by Proposition~\ref{OT}, to find a non-separating, untwisted, unknotted curve contained in $S$.  

Our construction involves surgery on an unknotted curve in $S$, so the first step is to record some results on different framings of the curve before and after surgery.

Let $M= (S,h)$ be an open book decomposition and let $C$ be a simple closed curve embedded in the interior of $S$ which bounds an embedded disk in $M$.  Let $N$ be a regular neighborhood of $C$. Note that although there is a unique canonical meridian $\mu$, there are two natural choices of longitude in this setting:
\begin{enumerate}
\item $\lambda_D$, the slope of the single curve $D\cap\partial N$
\item $\lambda_S$, the slope shared by the two curves $S\cap\partial N$.
\end{enumerate}
Using the orientation convention that $\langle \lambda_S, \mu\rangle = \langle \lambda_D, \mu \rangle = 1$ with respect to the outward pointing normal on $\partial N$, define the integer $tw_C(S)$, the {\it twisting of S along C}, by writing $\lambda_S=\lambda_D+tw_C(S)\mu$.

If the orientation for $\lambda_D$ is reversed, then $\mu$ and hence $\lambda_S$ also have their orientations reversed.  It follows that the sign of $tw_C(S)$ does not depend on the choice of orientation of $D$.  In particular, from the point of view of $D$, if the sign of $tw_C(S)$ is negative, an annular neighborhood of $C$ in $S$ twists to the left, when traveling around $\partial D$.  Thus $tw_C(S)$ serves as a topological stand-in for the Thurston-Bennequin invariant.

Perform a $p\mu+q\lambda_D$ surgery along $C$, for some relatively prime integers $p$ and $q$, and let $Y$ denote the manifold thus obtained from $M$.

\begin{prop}\label{surgery} The manifold $Y$ is homeomorphic to $M$ if and only if $p/q=\pm 1/q$.
\end{prop}

\begin{proof} Let $B$ denote a regular neighborhood of $D$ in $M$. Note that $M=M\sharp S^3$ with summing sphere $\partial B$. So it suffices to consider the case that $M=S^3$. If $p=\pm 1$, simply cut and twist around $D$ to realize the homeomorphism between $M$ and $Y$. Computing $\pi_1(Y)=<\mu | p\mu> = \mathbb Z_p$ shows the converse is also true.
\end{proof}

\begin{prop}\label{tauqofh} If $tw_C(S)=0$ and $p= - 1$, then $Y$ has an open book decomposition given by $(S, T_C^q h)$.  Moreover, the analogues of $C, D,$ and $S$ exist in $Y$ as embedded objects, and in $Y$ it remains the case that $tw_C(S)=0$. 
\end{prop}

\begin{proof} Let $A$ be an annular neighborhood of $C$ in $S$ and let $N=A\times[1/2,1]$. Let $b$ be non-separating embedded arc in $A$  and let $R = b \times [1/2,1]$ be a rectangular compressing disk for $N$.  

To construct the desired homeomorphism from $Y$ to the manifold built from $(S, T_C^q h)$, start with the identity map from $(S \times I)-N \to (S\times I)-N$.  Next extend the homeomorphism to $A \times \{1\} \to A\times\{1\}$ as follows.  Let $x,y\in A$.  At the quotient level, we have $(x,1)\sim (h(x),0)$ and $(y,1)\sim(T_C^q h(y),0)$, thus we define $(x,1)\mapsto (y,1)$ if $h(x)=T_C^q h(y)$.

\begin{figure}[ht]
  	\centering
   	\includegraphics[width=3.5truein]{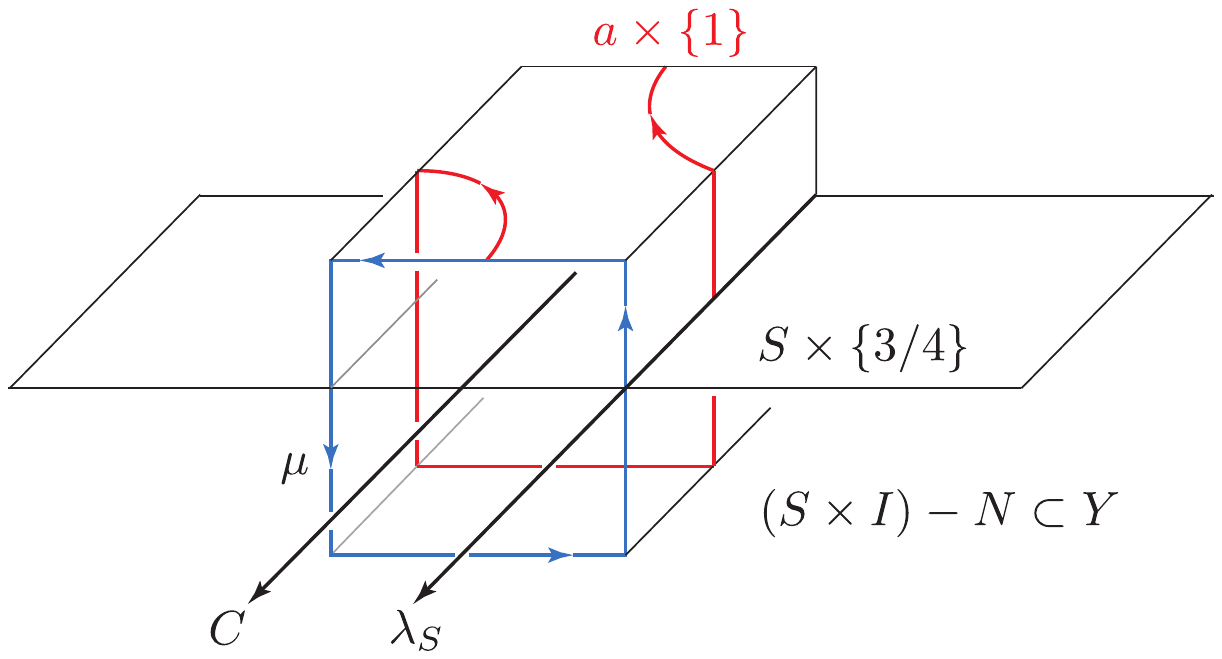} 
	\caption{}
	\label{surgery}
\end{figure}

Next we compute the preimage of $\partial R$ to compute the required filling of $\partial N$ in the domain.  To be explicit about orientations, pick an orientation for $C$.  Choose the orientation on $\lambda_S$ so that it is isotopic to $C$ in $N(C)$.  The orientation on $\lambda_S$ determines the orientation on $\mu$.  See Figure~\ref{surgery}. The only portion of $\partial R$ not mapped by the identity is the arc $b \times \{1\}$.  From the formula above, we see the preimage of $b \times \{1\}$ is an arc $a\times\{1\}$ in $A\times \{1\}$ such that $h(a)=T_C^q h(b)$, or equivalently, $a=h^{-1}T_C^q h(b)$.  It follows that the preimage of $\partial R$ in $M$ is a $\mu - q\lambda_S=\mu -q\lambda_D$ curve, thus the homeomorphism on $(S \times I)-N$ extends to a homeomorphism defined on all of $Y$.
\end{proof}

\begin{prop}\label{OT} If $C$ bounds an embedded disk $D$ in $M=(S,h)$, $tw_C(S)\ge 0$, and $C$ is non-separating in $S$, then the contact structure determined by $(S,h)$ is overtwisted.
\end{prop}

\begin{proof} Since $C$ is non-separating, it is non-isolating.  Thus by Honda's Legendrian realization principle \cite{Ho}, we may assume $C$ is Legendrian.  The Thurston-Bennequin invariant of $C=\partial D$ is equal to $tw_C(S)$, thus if $tw_C(S)=0$, $D$ is an overtwisted disk.  In the remaining cases, $C$ can be isotoped to decrease the Thurston-Bennequin invariant and thereby produce an overtwisted disk.
\end{proof}

\begin{cor}\label{qsurgery} If $C$ bounds an embedded disk $D$ in $M=(S,h)$, $tw_C(S)= 0$, and $C$ is non-separating in $S$, then the contact structure determined by $(S,T_C^q h)$ is overtwisted for any integer $q$.
\end{cor}

\begin{proof} The invariant $tw_C(S)$ has two possible interpretations depending on whether it is computed in $M=(S,h)$ or in $Y=(S,T_C^q h)$.

In $M$ we have $\lambda_S=\lambda_D + tw_C(S)\mu$. By Proposition~\ref{tauqofh} we have in $Y$, $\mu'=\mu - q\lambda_D$ and consequently, $\lambda_S=\lambda_D + tw_C(S)(\mu'+q\lambda_D)=(1 + tw_C(S)q)\lambda_D + (\mu' tw_C(S))$.  When $tw_C(S)=0$ in $M$, $\lambda_S=\lambda_D$.  This formula is satisfied on just the boundary of the regular neighborhood of $C$, thus it holds in $Y$ as well.
\end{proof}

\begin{example}\label{cable_8_20}  Let $F$ be the Seifert surface for the knot $8_{20}$ as described in Example~\ref{Example_8_20}.  Using the notation of Section~\ref{pA}, form the $(2,1)$-cable of $F$, and denote its Seifert surface by $S = F_0\cup P\cup F_{1/2}$.  This gives an open book decomposition $(S, H_0')$ of $S^3$. To produce the curve $C$ used in the definition of $H = T_C \circ H_0$ consider the curves $A$ and $B$ that live on $F$ as shown in Figure~\ref{8_20}.

Pick an arc $a \subset F_{0}$ that runs from a point of $A \times \{0\}$ to the point where the twisted band is added in the formation of the $(2,1)$-cable.  Choose an arc $b \subset F_{1/2}$ similarly corresponding to $B \times \{1/2\}$. In addition the arcs should be chosen so that $A, B$, and the projections of $a,b$ to $F$ have no interior intersections.  Define $C$ to be the band sum, in $S$, of $A \times \{0\}$ and $B \times \{1/2\}$ along an arc consisting of $a, b$, and an arc that runs across the twisted band.  

Since $A$ and $B$ are non-separating in $F$,   $C$ does not separate when restricted to either $F_{0}$ or $F_{1/2}$ 
and  Theorem~\ref{H} applies. Combining Theorem~\ref{H}, Proposition~\ref{RV}, Corollary~\ref{H'inS3}, and Proposition~\ref{OT} shows that $H$ is pseudo-Anosov, right-veering with fractional Dehn twist coefficient 1/2, and the associated contact structure is overtwisted.
\end{example}

\begin{thm}\label{counterexamples} There exists an infinite collection of fibred knots in $S^3$ such that the monodromy is pseudo-Anosov, right-veering with fractional Dehn twist coefficient 1/2, and the associated contact structure is overtwisted.
\end{thm}

\begin{proof} It is enough, by the construction of Example~\ref{cable_8_20}, to produce knots with pseudo-Anosov monodromy whose fibres contain two disjoint Hopf bands of opposite signs.  One such family of examples are the nontorus fibred 2-bridge knots of genus 2 and higher.  The fibres of such knots correspond to words in L and R that use both letters and have even length \cite{Sch, GK}.  The letters correspond to Hopf bands plumbed together in a vertical stack.  As long as the surface has genus at least 2, the word will have length at least 4.  The desired Hopf bands correspond to a non-adjacent pair of letters L and R which necessarily exist.  
\end{proof}

If Conjecture~\ref{main1} is true, then none of the examples of Theorem~\ref{counterexamples} and Example~\ref{cable_8_20} can be destabilized, and thus they would all provide additional counterexamples to the conjecture of \cite{HKM1}.


\begin{thebibliography}{tralala}

\bibitem[CB]{CB} A.\ Casson, S.\ Bleiler, \textit{Automorphisms of surfaces after Nielsen and Thurston}, London Mathematical Society Student Texts, {\bf 9}. Cambridge University Press, Cambridge, 1988. iv+105 pp.

\bibitem[CH]{CoHo} V.\ Colin, K.\ Honda, \textit{Reeb vector fields and open book decompositions}, \texttt{ArXiv:/0809.5088v1}

\bibitem[GO]{GO} D.\ Gabai, U.\  Oertel, \textit{Essential laminations in 3-manifolds} Ann.\  of Math. {\bf (2) 130} (1989), no. 1, pp. 41--73.

\bibitem[Ga1]{Gagen} D. \ Gabai, \textit{Foliations and Genera of Links}, Topology {\bf (23)} (1984),  381--394.

\bibitem[Ga2]{Gafib} D. \ Gabai, \textit{Detecting fibred links in $S^3$}, Comm.\ Math.\ Helv. \ {\bf (61)} (1986),  519--555.

\bibitem[Ga3]{Ga3} D. \ Gabai, \textit{Problems in foliations and laminations}. Geometric topology (Athens, GA, 1993), 1--33, AMS/IP Stud. Adv. Math., 2.2, Amer. Math. Soc., Providence, RI, 1997.

\bibitem[Ga4]{Ga2} Personal communication.

\bibitem[GK]{GK} D. Gabai and W.H. Kazez, \textit{Pseudo-Anosov maps and surgery on fibred 2-bridge knots}, Topology and its Applications, {\bf  37} (1990),  93--100.

\bibitem[Gi]{Gi} E. Giroux, {\it G\'eom\'etrie de contact: de la dimension trois vers les dimensions sup\'erieures}, Proceedings of the International Congress of Mathematicians, Vol. II (Beijing, 2002), 405�414, Higher Ed. Press, Beijing, 2002.

\bibitem[H]{Ho} K. Honda, {\it On the classification of tight contact structures I}, Geom. Topol. {\bf 4} (2000), pp. 309�368.

\bibitem[HKM1]{HKM1} K.\ Honda, W.\ Kazez and G.\ Mati\'c, \textit{Right-veering diffeomorphisms of compact surfaces with boundary}, Invent.\ Math.\ {\bf 169}, (2007), pp 427--449.

\bibitem[HKM2]{HKM2} K.\ Honda, W.\ Kazez and G.\ Mati\'c, \textit{Right-veering diffeomorphisms of compact surfaces with boundary II}, Geom.\ Topol.\ {\bf 12} (2008), 2057--2094.

\bibitem[IK]{IK} T.\ Ito, K.\ Kawamuro, \textit{Open book foliation}, \texttt{ArXiv:1112.5874v1}

\bibitem[Le]{Lekili} Y.\ Lekili, {\it Planar open books with four binding components }, \texttt{ArXiv:1008.3529}

\bibitem[Li]{Lisca} P.\ Lisca, {\it On overtwisted, right-veering open books}, \texttt{ArXiv:1107.5268}.

\bibitem[Ro1]{Ro1} R.\ Roberts, \textit{Taut foliations in punctured surface bundles, I}, Proc.\ London Math.\ Soc.\ (3)  {\bf 82}  (2001), 747--768.

\bibitem[Ro2]{Ro2} R.\ Roberts, \textit{Taut foliations in punctured surface bundles, II}, Proc.\ London Math.\ Soc.\ (3) {\bf 83} (2001), 443--471.

\bibitem[Rolf]{Rol} D. \ Rolfsen, {\it Knots and links}, Mathematics Lecture Series 7, Publish or Perish, Berkeley, CA (1976).


\bibitem[Sch]{Sch} H.\ Schubert,  {\it \"Uber eine numerische Knoteninvariante}, Math. Z. {\bf 61} (1954), 245--288.

\bibitem[Seif]{Seifert} H.\ Seifert, \textit{Topologie dreidimensionaler gefaserter R\"{a}ume},  Acta \ Math.\  {\bf 60} (1933), 147--238.

\bibitem[St]{St} J.\ Stallings, \textit{Constructions of fibred knots and links}, Algebraic and geometric topology (Proc. Sympos. Pure Math., Stanford Univ., Stanford, Calif., 1976), Part 2, pp. 55--60, Proc. Sympos. Pure Math., XXXII, Amer. Math. Soc., Providence, R.I., 1978.

\bibitem[Th]{Th} W.\ Thurston \textit{On the geometry and dynamics of diffeomorphisms of surfaces}, Bull.\ Amer.\ Math. \ Soc.   {\bf 19}, (1988), pp 417--431.

\end{thebibliography}
\end{document}